\documentclass[11pt]{article}

\usepackage{graphicx}

\usepackage{amsmath,amssymb,amsthm}
\usepackage{array}
\usepackage{mdwmath}
\usepackage{mdwtab}
\usepackage[font=footnotesize,caption=false]{subfig}
\usepackage{url}

\usepackage{hyperref}
\usepackage[margin=3cm]{geometry}
\usepackage{eucal}
\usepackage{bm}
\usepackage{enumerate}
\usepackage{psfrag}

\newcommand{\reals}{\mathbb{R}}
\newcommand{\pr}{\mathbb{P}}
\newcommand{\eps}{\varepsilon}
\newcommand{\ex}{\mathbb{E}}
\newcommand{\Xb}{\mathbf{X}}
\newcommand{\Xbs}{\mathbf{X}_*}
\newcommand{\Dpk}{D_\phi^k}
\newcommand{\pipk}{\pi_\phi^k}

\newcommand{\rhob}{\overline{\rho}}

\newcommand{\epseq}{\stackrel{\eps}{\asymp}}
\newcommand{\epsle}{\stackrel{\eps}{\preccurlyeq}}

\newcommand{\kb}{\bm{k}}
\newcommand{\Sb}{\bm{S}}
\newcommand{\betab}{\bm{\beta}}

\newcommand{\realspp}{\reals_{++}}
\newcommand{\mixR}{\mathcal{M}}
\newcommand{\llr}{D}

\newcommand{\ical}{\mathcal{I}}
\newcommand{\scal}{\mathcal{S}}

\newcommand{\ms}{m_*}
\newcommand{\ks}{k_*}

\newcommand{\Mterm}{M}
\newcommand{\Sterm}{S}
\newcommand{\Lterm}{L}
\newcommand{\Tterm}{T}
\newcommand{\Ttermt}{\widetilde{\Tterm}}

\newcommand{\nats}{\mathbb{N}}
\newcommand{\taut}{\widetilde{\tau}}
\newcommand{\rclass}{\Delta} 
\newcommand{\Ss}{\mathcal{S}}

\newcommand{\Fc}{\mathcal{F}}

\newcommand{\Dphi}{D_\phi}
\newcommand{\ceil}[1]{\lceil #1 \rceil}
\newcommand{\Tt}{\widetilde{T}}
\DeclareMathOperator{\supp}{supp}
\newcommand{\asyma}[1]{\stackrel{#1}{\asymp}}
\newcommand{\rr}{r}

\newcommand{\kcal}{\mathcal{K}}
\newcommand{\lcal}{\mathcal{L}}

\newcommand{\Et}{\widetilde{E}}
\newcommand{\floor}[1]{\lfloor #1 \rfloor}

\newcommand{\ei}{e}
\newcommand{\gam}{\gamma}
\newcommand{\neib}{\partial}
\newcommand{\pit}{\widetilde{\pi}}

\newcommand{\info}{I}
\newcommand{\tauh}{\widehat{\tau}}

\newcommand{\vecX}{\mathbf{X}}

\newenvironment{itemize*}%
  {\begin{itemize}%
    \setlength{\itemsep}{0pt}%
    \setlength{\parskip}{0pt}}%
  {\end{itemize}}

\newcommand{\lam}{\lambda}
\newcommand{\lams}{\lambda_*}
\newcommand{\Rule}[1]{(R#1)}

\newcommand{\mui}{u}
\newcommand{\nui}{\nu}
\newcommand{\cn}{}

\newtheorem{thm}{Theorem}

\newtheorem{lem}{Lemma}
\newtheorem{defn}{Definition}

\begin{document}

\title{Sequential detection of multiple change points in
networks: a graphical model approach}

\author{Arash Ali Amini and XuanLong Nguyen \\
Department of Statistics, University of Michigan
\thanks{
A preliminary version of
this paper was presented at the International Symposium
on Information Theory, Boston, Massachusetts, July 2012~\cite{NguAmiRaj12}. 
We would like to thank
Ram Rajagopal for valuable discussions and help
during the course of this research. This work was supported 
in part by NSF grants CCF-1115769 and OCI-1047871.}
}

\maketitle

\begin{abstract}

We propose a probabilistic formulation that enables sequential
detection of multiple change points in a network setting. We
present a class of sequential detection rules for certain functionals of 
change points (minimum among a subset), and prove their asymptotic optimality properties
in terms of expected detection delay time.
Drawing from graphical model formalism, the sequential detection
rules can be implemented by a computationally efficient 
message-passing protocol which may scale up linearly in network size
and in waiting time. The effectiveness of our inference algorithm is demonstrated by simulations.
\end{abstract}

\section{Introduction}

Classical sequential detection is the problem of detecting 
changes in the distribution of data collected sequentially
over time~\cite{Lai-01}. In a decentralized
network setting, the decentralized sequential detection
problem concerns with data sequences 
aggregaged over the network, while sequential detection
rules are constrained to the network structure (see, e.g.,
\cite{Veeravalli-etal-93,Hus94,Mei-08,Nguyen-Wainwright-Jordan-08,FelMou11}).
The focus was still on a \emph{single}
change point variable taking values in (discrete) time.
In this paper, our interests lie in sequential detection in a network
setting, where multiple change point variables may be simultaneously
present. 

As an example, quickest
detection of traffic jams concerns with
multiple potential hotspots (i.e., change points) spatially located
across a highway network. A simplistic approach is to
treat each change point variables independently,
so that the sequential analysis of individual change points can
be applied separately. However, it
has been shown that accounting for the statistical dependence among
the change point variables can provide significant improvement
in reducing both false alarm probability and detection delay 
time~\cite{Rajagopal-etal-08}.

This paper proposes a general probablistic formulation for the multiple
change point problem in a network setting, 
adopting the perspective of probabilistic graphical
models for multivariate data \cite{Jordan-Statsci-04}. We consider
estimating functionals of multiple change points defined globally and 
locally across the network.
The probablistic formulation enables the borrowing of statistical
strengh from one network site (associated with a change point variable) 
to another.
We propose a class of sequential detection rules, which can
be implemented in a message-passing and distributed fashion across
the network. The computation of the proposed sequential rules
scales up linearly in both network size and in waiting time,
while an approximate version scales up constantly in waiting time.
The proposed detection rules are shown to be asymptotically 
optimal in a Bayesian setting. Interestingly, 
the expected detection delay can be 
expressed in terms of Kullback-Leibler divergences 
defined along edges of the network structure. 
We provide simulations that demonstrate both 
statistical and computational efficiency of our approach.

\medskip

\noindent {\bf Related Work.} 
The rich statistical literature on sequential analysis tends to focus
almost entirely on the inference of a single change point 
variable~\cite{Lai-01,VerBan12}. 
There are recent formulations for sequential diagnosis of 
a single change point,
which may be associated with multiple causes \cite{Dayanik-etal-08}, or
multiple sequences \cite{Xie-Siegmund-11}. 
Another approach taken in~\cite{RagVee10} considers a change propagating 
in a Markov fashion across an array of sensors. 
These are interesting directions but 
the focus is still on detecting the onset of a single event.  
Graphical models have been considered
for distributed learning and decentralized detection before, but not 
in the sequential setting~\cite{Cetin-etal-06,Kreidl-Willsky-07}. 
This paper follows the line of 
work of~\cite{Rajagopal-etal-08,Rajagopal-etal-10}, but our
formulation based on graphical models is more general, and we impose 
less severe constraints on the amount of information that can be exchanged
across network sites.

\medskip
\noindent {\bf Notation.} We will use $P$ to denote  densities w.r.t. some underlying measure (usually understood from the context), while $\pr$ is used to denote probability measures. $[d]$ denotes the set of integers $\{1,\dots,d\}$. For a real-valued function $f$ defined on some space, $\|f\|_\infty := \sup_x |f(x)|$ denotes its uniform norm. In an undirected graph, the neighborhood of a node $i$ is denotes as $\partial i$.

\section{Graphical model for multiple change points }
In this section, we shall formulate the multiple change point detection
problem, where the change point variables and observed data
are linked using a graphical model. 
Consider a sensor network with $d$ sensors, each of which is associated
with a random variable $\lambda_j \in \nats$, for $j \in [d] := \{1,2,\dots,d\}$,
representing a \emph{change point}, the time at which a sensor fails
to function properly. We are interested in detecting
these change points as accurately and as early as possible, using 
the data that are associated with (e.g., observed by) the sensors.
Taking a Bayesian approach, each $\lambda_j$ is independently endowed
with a prior distribution $\pi_j(\cdot)$.

 A central ingredient in our
formalism is the notion of a \emph{statistical graph}, denoted as $G = (V,E)$, which specifies the
probabilistic linkage between the change point variables and
observed data collected in the network (cf. Fig.~\ref{Fig-gm}). The vertex set of the graph, $V
= [d]$ represents the indices of the change point variables
$\lambda_j$. The edge set $E$ represents pairings of change point
variables, $E = \{ \ei = \{s_1,s_2\} \mid s_1,s_2 \in V\}$. With each
vertex and each edge, we associate a sequence of \emph{observation}
variables,
\begin{align}
  \Xb_j = (X_j^1,X_j^2,\dots), \quad j \in V, \\
  \Xb_e = (X_\ei^1,X_e^2,\dots), \quad \ei \in E,
\end{align}
where the superscript denotes the time index. The $\Xb_j$ models the
private information of node $j$, while $\Xb_e$ models the shared
information of nodes connected by $\ei$. We will use the notation
$\Xb_j^n = (X_j^1,\dots,X_j^n)$ and similarly for $\Xb_\ei^n$;
notice the distinction between $X_j^n$, the observation at time $n$,
versus bold $\Xb_j^n$, the observations up to time $n$, both at node
$j$. The aggregate of all the observations in the network is denoted
as $\Xbs = (\Xb_j,j\in V, \Xb_\ei, \ei \in E)$. Similarly, $\Xbs^n$
represents all the observations up to time $n$. We will also use
$\lams = (\lam_j,j \in V)$.

\begin{figure*}[t]
\psfrag{Lambda1}{$\lambda_1$}
\psfrag{Lambda2}{$\lambda_2$}
\psfrag{Lambda3}{$\lambda_3$}
\psfrag{Lambda4}{$\lambda_4$}
\psfrag{Lambda5}{$\lambda_5$}
\psfrag{X1}{$\vecX_{12}$}
\psfrag{X2}{$\vecX_{23}$}
\psfrag{X3}{$\vecX_{34}$}
\psfrag{X4}{$\vecX_{45}$}
\psfrag{m12}{$m_{12}^{n}$}
\psfrag{m24}{$m_{24}^{n}$}
\psfrag{m32}{$m_{32}^{n}$}
\psfrag{m45}{$m_{45}^{n}$}
\centerline{
\begin{tabular}{ccc}
\includegraphics[width=.32\textwidth]{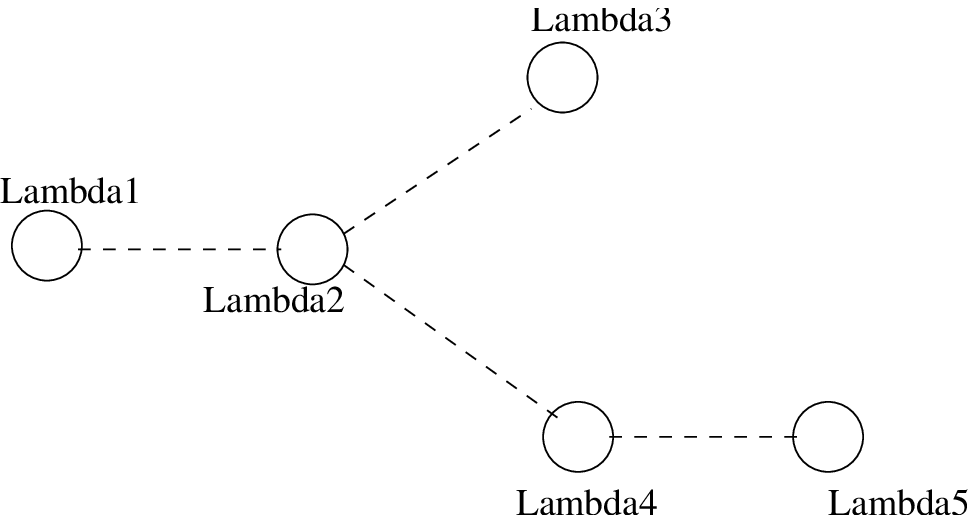} &
\includegraphics[width=.32\textwidth]{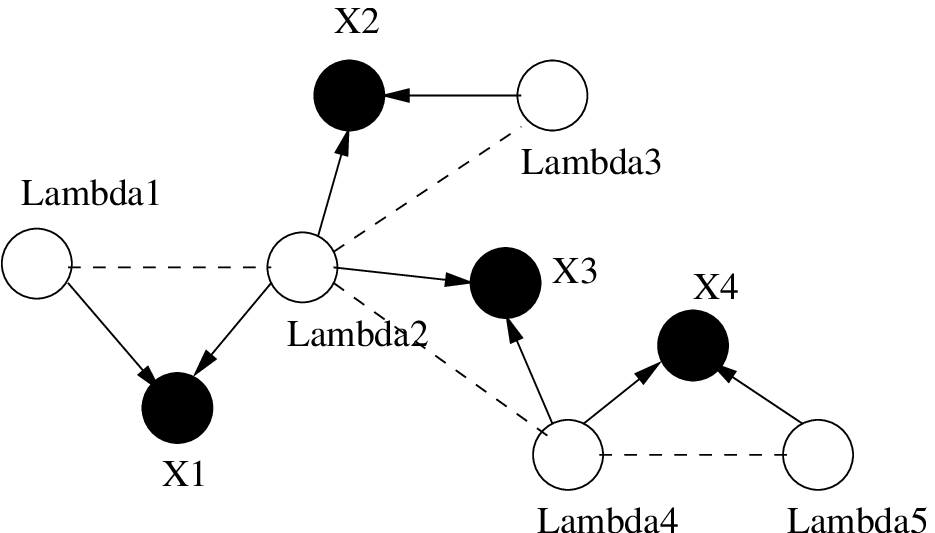} &
\includegraphics[width=.32\textwidth]{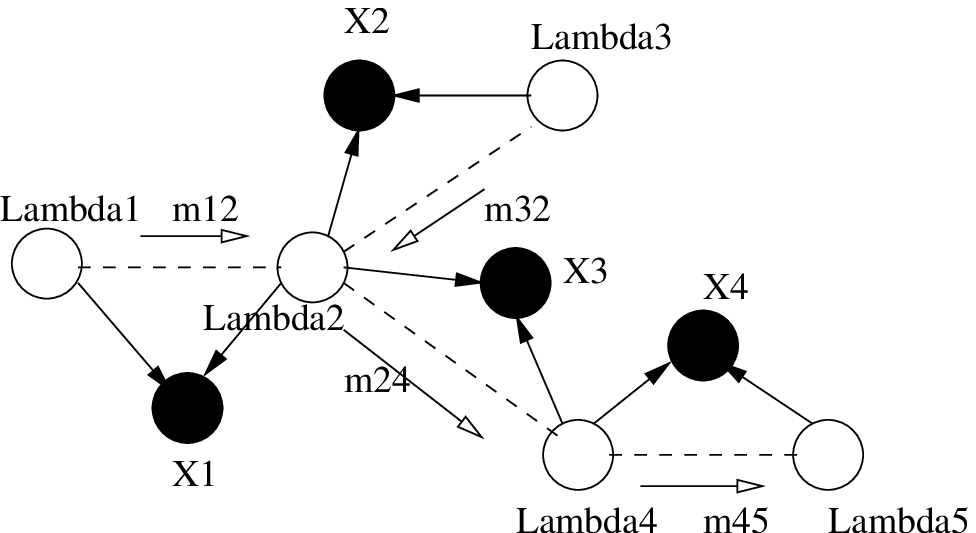}
\end{tabular}}
\caption{Left panel illustrates a statistical graph,
which induces a graphical model in the middle panel. Right panel
illustrates statistical messages passed at time $n$ along some edges
in a communication graph (which coincides with statistical graph in this
case).}
\label{Fig-gm}
\end{figure*}

The joint distribution of $\lams$ and $\Xbs^n$ is given by a graphical
model,
\begin{align}\label{eq:joint:def}
  P(\lams,\Xbs^n) = \prod_{j \in V} \pi_j(\lam_j) \prod_{j \in
    V} P(\Xb_j^n | \lam_j) \prod_{\ei\, \in E} P(\Xb_\ei^n|\lambda_{s_1},\lambda_{s_2}).
\end{align}
Given $\lam_j = k$, we assume $X_j^1,\dots,X_j^{k-1}$ to be
i.i.d. with density $g_j$ and $X_j^{k},X_j^{k+1},\dots$ to be
i.i.d. with density $f_j$. Given $(\lam_{s_1},\lam_{s_2})$, we assume
that the distribution of $\Xb_e^n$ only depends on $\lambda_e := \lambda_{s_1}
\wedge \lambda_{s_2}$, the minimum of the two change points; hence we often write $P(\Xb_e^n | \lam_e)$
instead of $P(\Xb_\ei^n|\lambda_{s_1},\lambda_{s_2})$. Given $\lam_e =
k$, $X_e^1,\dots,X_e^{k-1}$ are i.i.d. with density $g_e$ and
$X_e^k,X_e^{k+1},\dots$ are i.i.d. with density $f_e$. All the
densities are assumed to be with respect to some underlying measure
$\mu$. These specifications can be summarized as,
\begin{align}
   P(\Xb_j^n | \lam_j) = \prod_{t = 1}^{k-1} g_j(X_j^t) \prod_{t=k}^n f_j(X_j^t)
\end{align}
and similarly for $P(\Xb_e^n | \lam_e)$. We will assume the prior on
$\lam_j$ to be geometric with parameter $\rho_j \in (0,1)$,
i.e. $\pi_j(k) := (1-\rho_j)^{k-1} \rho_j$, for $k \in \nats$.
Note that these change point variables are 
dependent a posteriori, despite being independent a priori.

\subsection{Sequential rules and optimality}
Although our primary interest is in sequential estimation of the
change points $\lams = (\lambda_j)$, we are in general interested in the
following functionals,
\begin{align}
  \phi := \phi(\lams) := \lam_\Ss := \min_{j \in \Ss} \lam_j.
\end{align}
for some subset $\Ss \subset [d]$. Examples include a single change point $\Ss = \{j\}$, the earliest
among a pair $\Ss =\{i,j\}$ and the earliest in the entire network
$\Ss = [d]$. Let $\Fc_n = \sigma(\Xbs^n)$ be the $\sigma$-algebra
induced by the sequence $\Xbs^n$. A sequential detection rule for
$\phi$ is formally a stopping time $\tau$ with respect to filtration
$(\Fc_n)_{n \ge 0}$. To emphasize the subset $\Ss$, we will use
$\tau_\Ss$ to denote a rule when the functional $\phi = \lam_\Ss$. For
example $\tau_1$ is a detection rule for $\lambda_1$ and $\tau_{12}$
is a rule for $\lam_{12} = \lam_1 \wedge \lam_2$.

In choosing $\tau$, there is a trade-off between the
false alarm probability $\pr(\tau \le \phi)$ and the detection delay
$\ex (\tau - \phi)_+$. Here, we adopt
the Neyman-Pearson setting to consider
all stopping rules for $\phi$, having false alarm at most $\alpha$,
\begin{align}
  \rclass_\phi(\alpha) := \{ \tau : \;\pr(\tau \le \phi) \le \alpha \},
\end{align}
and pick a rule in $\rclass_\phi$ that has minimum
detection delay. 



\subsection{Communication graph and message passing (MP)}
Another ingredient of our formalism is the notion of a
\emph{communication graph} representing constraints under which the
data can be transmitted across network to compute a particular
stopping rule, say $\tau_j$. In general, such a  rule depends
on all the aggregated data $\Xbs^n$. We are primarily interested in those rules
that can be implemented in a distributed fashion by passing messages
from one sensor only to its neighbors in the communication graph.
Although, conceptually, the statistical graph and communication graphs play two
distinct roles, they usually coincide in practice and this will be
assumed throughout this paper. See Fig.~\ref{Fig-gm} for an
illustration.



\section{Proposed stopping rules}

In general, we suspect that obtaining strictly optimal rules in closed form is not possible for the multiple
change point problem introduced earlier; more crucially such rules are not computationally tractable for large networks.  In this section, we shall present a class of detection rules that scale linearly in the
size of the network, $d$,
and
can be implemented in a distributed fashion by message passing.

Consider the following posterior probabilities
\begin{align}
  \gam_\Ss^n(k) &:= \pr (\lam_\Ss = k \mid\Xbs^n),  \label{eq:post:k} \\
  \gam_\Ss^n[n] &:= \pr (\lam_\Ss \le n \mid \Xbs^n) = \sum_{k=1}^n
  \gam_\Ss^n(k). \label{eq:post:upto:n}
\end{align}
We propose to stop at the first time $\gam_\Ss^n[n]$ goes above a
threshold,
\begin{align}\label{eq:tau:S:def}
  \tau_\Ss = \inf \{ n \in \nats: \; \gam_\Ss^n[n] \ge 1 -\alpha \}
\end{align}
where $\alpha$ is the maximum tolerable false alarm. It is easily
verified that these rules have a false alarm at most $\alpha$.
\begin{lem}
  For $\phi = \lam_\Ss$, the rule $\tau_\Ss \in \rclass_\phi(\alpha)$.
\end{lem}
More interestingly, we will show in Section~\ref{sec:optimal:delay:thm} that $\tau_\Ss$ is asymptotically optimal for detecting $\lam_\Ss$. First, let us look at two message-passing (MP) implementations of the stopping rule~\eqref{eq:tau:S:def}.

\subsection{Exact message passing algorithm}\label{sec:exact:alg}
It is relatively simple to adapt the well-established belief propagation
algorithm, also known as sum-product, to the graphical
model~(\ref{eq:joint:def}). The algorithm produces exact values of the
posterior $\gam_\Ss^n$, as defined in~(\ref{eq:post:k}), in the cases
where $G$ is a polytree (and provides a reasonable estimate
otherwise.) In this section, we provide the details for $\Ss =
\{j\}$ or $\Ss = \{i,j\} \in E$.

One issue in adapting the algorithm is the possible infinite support
of $\gam_\Ss^n$. Thanks to a ``constancy'' property of the likelihood,
it is possible to lump all the states after $n$ when computing
$\gam_\Ss^n[n]$.
\begin{lem}\label{lem:const}
  Let $\{i_1,i_2,\dots,i_r\} \subset [d]$ be a distinct collection of
  indices. The function 
  \begin{align*}
   (k_1,k_2,\dots,k_r) \mapsto
  P(\Xbs^n|\lam_{i_1}=k_1, \lam_{i_2}=k_2,\dots,\lam_{i_r} = k_r) 
  \end{align*}
   is  constant over $\{n+1,n+2,\dots\}^r$.
\end{lem}
See Appendix~\ref{app:lem:const} for the proof.
The algorithm is invoked at each time step $n$, by passing messages
between nodes according to the following protocol: a node sends a
message to one of its neighbors (in $G$) when and only when it
has received messages from all its other neighbors. Message passing
continues until 
any node can be linked to any other node by a chain of messages,
assuming a connected graph. For a tree, this
is usually achieved by designating a node as root and passing messages
from the root to the leaves and then backwards.

The message that node $j$ sends to its neighbor $i$, at time $n$, is
denoted as $m_{ji}^n = [m_{ji}^n(1),\dots,m_{ji}^n(n+1)] \in
\reals^{n+1}$ and computed as
\begin{align}\label{eq:mji:forumla}
  m_{ji}^n(k) = \sum_{k'=1}^{n+1} \Big\{ \pit_j^n(k') P(\Xb_j^n|k')
  P(\Xb_{ij}^n|k\wedge k') \!\! \prod_{r \in
    \neib j \setminus\{i\}} \!\!m_{rj}^n(k') \Big\} 
\end{align}
for $k \in [n+1]$, where 
\begin{align}\label{eq:pit:def}
    \pit_j^n(k) := 
        \begin{cases}
        \pi_j(k) & \text{for $k \in [n]$} \\
        \pi_j[n]^c = \sum_{k=n+1}^\infty \pi_j(k)
        & \text{for $k=n+1$}.
        \end{cases}
\end{align}
 and
$\neib j$ is the neighborhood set of $j$. Once the message passing
ends, $\gam_j^n$ and $\gam_{ij}^n$ are readily available. We have
\begin{align}\label{eq:gamj:formula}
  \gam_j^n(k) \propto \pit_j^n(k)\, P(\Xb_j^n|k) \prod_{r \in \partial j}\!\!
  m_{rj}^n(k), \quad k \in [n].
\end{align}
It also holds for $k=n+1$ if the LHS is interpreted as
$\gam_j^n[n]^c$.

The same messages can be used to compute $\zeta_{ij}^n(k_1,k_2) := P(\lam_i=k_1,\lam_j=k_2 | \Xbs^n)$ for $\{i,j\} \in E$. We have
\begin{align}\label{eq:joint:post:formula}
     \zeta_{ij}^n(k_1,k_2) &\propto \; 
     \psi_{ij}^n(k_1,k_2)     \prod_{r \in \partial i \setminus \{j\}}\!\! m_{ri}^n(k_1)
 \prod_{r \in \partial j \setminus \{i\}}\!\! m_{rj}^n(k_2) \end{align}
where
\begin{align}\label{eq:joint:psi:def}
     \psi_{ij}^n(k_1,k_2) &:= \pit_i^n(k_1)\, \pit_j^n(k_2)\, P(\Xb_i^n|k_1)\, P( \Xb_j^n|k_2) \,
  P(\Xb_{ij}^n|k_1 \wedge k_2)
\end{align}
for $(k_1,k_2) \in [n]^2$, from which $\gam_{ij}^n$ can be computed. 

Let us summarize the steps of the message passing algorithm in the case of a tree:
\newpage

\medskip
\noindent \rule{\textwidth}{1pt}
Message passing algorithm to compute the posteriors $\gam_j^n[n]$ and $\gam_{ij}^n[n]$ 
\vskip-1ex
\noindent \rule{\textwidth}{1pt}
At time each time $n$: 
\begin{enumerate}
    \item Designate a node of the tree, say node $1$ as root and direct the edges to point away from root.
    \item Initialize messages $m_{ji}^n \in \reals^{n+1}$ (one for each directed edge $j \to i$) to the all ones vector. Compute $\pit_j^n(k)$ for $k \in [n+1], j \in [d]$ according to~\eqref{eq:pit:def}.
    \item Pass messages $m_{ji}^n$ from a node $j$ to each of its descendants $i$ (that is, $i \in \partial j$ for which $j \to i$ is a directed edge.) according to equation~\eqref{eq:mji:forumla}. Do this, recursively, starting from root ($j=1$) until you reach each of the leaves.
    \item Reverse the direction of the edges and repeat Step~3, this time starting from leaves and ending once you reached the root. In computing $m_{ji}^n$ based on~\eqref{eq:mji:forumla}, use messages computed in the previous step.
    \item Compute $\gam_j^n(k)$ for  $k \in [n+1]$ based on~\eqref{eq:gamj:formula} and normalize so that $\sum_{k=1}^{n+1} \gam_j^n(k) = 1$. Let $\gam_j^n[n] = \sum_{k=1}^n \gam_j^n(k)$.
    \item Compute $\zeta_{ij}^n(k_1,k_2)$ for $(k_1,k_2) \in [n+1]^2$ based on~\eqref{eq:joint:post:formula} and~\eqref{eq:joint:psi:def} and normalize so that $\sum_{k_1=1}^{n+1} \sum_{k_2 = 1}^{n+1} \zeta_{ij}^n(k_1,k_2) = 1$. Let $\gam_{ij}^n[n] := \sum_{k_1=1}^n \sum_{k_2 =1}^n \zeta_{ij}^n(k_1,k_2)$.
    \end{enumerate}
\noindent \rule{\textwidth}{1pt}

We have the following guarantee which is a restatement of a well-known result for belief propagation~\cite{Pearl88}:
\begin{lem}
  When $G$ is a tree, the message passing algorithm above produces correct
  values of $\gam_j^n$ and $\gam_{ij}^n$ at time step $n$, with
  computational complexity $O((|V| + |E|)n)$.
\end{lem}

\section{Asymptotic optimality of MP rules}
\label{sec:optimal:delay:thm}
This section contains our main result on the asymptotic optimality of stopping rule~\eqref{eq:tau:S:def}. To simplify the statement of the results, let us extend the edge set to
$\Et := E \cup \{\{j\}: j \in V\}$. This allows us to treat the
private data associated with node $j$, i.e. $\Xb_j$, as (shared) data
associated with a self-loop in the graph $(V,\Et)$.
For any $e \in \Et$, let
  $\info_e  := \int f_e \log \frac{f_e}{g_e} d\mu$
be the KL divergence between $f_e$ and $g_e$. For $\phi = \lam_\Ss$,
let
\begin{align}\label{eq:info:phi:def}
  I_\phi := I_{\lam_\Ss}:= \sum_{e\, \subset\, \Ss} I_{e}
\end{align}
where the sum runs over all $e \in \Et$ which are subsets of
$\Ss$. For example, for a chain graph on $\{1,2,3\}$ with node $2$
in the middle, $\Et =
\{\{1,2\},\{2,3\},\{1\},\{2\},\{3\}\}$ and we have $I_{\lam_{12}} :=
I_1 + I_2 + I_{12}$ while $I_{\lam_{13}} := I_1 + I_3$. (Here, we
abuse notation to write $I_{12}$ instead of $I_{\{1,2\}}$ and so on.)

Recall the geometric prior on $\lam_j$ (with parameter $\rho_j$)
and the definition of $\phi = \lam_\Ss$ as the minimum of $\lam_j, j
\in \Ss$. Then, $\phi$ is geometrically distributed a priori with
parameter
 $ 1 - e^{-q_\phi} := 1 - \prod_{j \in \Ss} (1 - \rho_j)$.

We can now state our main result on asymptotic optimality.

\newpage
\begin{thm}(Optimal delay) \label{thm:asym:opt}
 Assume $ \|\log \frac{f_e}{g_e}\|_\infty \le M$ for all $e
  \in \Et$, and geometric priors for $\{\lam_j\}$. Then, $\tau_\Ss$ is asymptotically optimal for $\phi =
  \lam_\Ss$; more specifically, as $\alpha \to 0$,
  \begin{align*}
    \ex \big[ \tau_\Ss - \lam_\Ss \mid \tau_\Ss \ge \lam_\Ss \big] &=
    \frac{|\log \alpha|}{ q_{\lam_\Ss} + I_{\lam_\Ss}} (1 + o(1)) \\
    &= \inf_{\taut \,
    \in \rclass_\phi(\alpha)} 
  \ex \big[ \taut - \lam_\Ss \mid \taut \ge \lam_\Ss \big].
  \end{align*}
\end{thm}

\noindent \emph{Remark 1.} Let us highlight some particular cases of interest in this result. To simplify notation, let $\rhob_j := 1-\rho_j$.
\begin{itemize}
    \item For $\phi = \lam_1 \wedge \cdots \wedge \lam_d$ (the minimum of all the change points), the asymptotic optimal delay is
        \begin{align*}
            \frac{|\log \alpha|}{- \sum_{j \in V} \log \rhob_j + \sum_{j \in V}I_j + \sum_{e \in E} I_e} (1+o(1))
        \end{align*}
   \item For $\phi = \lam_i \wedge \lam_j$, the asymptotic optimal delay is
   \begin{align*}
       \frac{|\log \alpha|}{-\log \rhob_i -\log \rhob_j + I_i + I_j + I_{ij} 1_{\{ \{i,j\} \in E\}}} (1 + o(1))
   \end{align*}
   where $1_{\{ \{i,j\} \in E\}}$ is an indicator function, i.e.,  equal to 1 if $\{i,j\}$ is an edge and zero otherwise.
   \item For $\phi = \lam_i$, the asymptotic optimal delay is
   \begin{align*}
       \frac{|\log \alpha|}{-\log \rhob_i + I_i} (1+o(1))
   \end{align*}
\end{itemize}

\noindent \emph{Remark 2.} A particular feature of the asymptotic delay is the decomposition~(\ref{eq:info:phi:def}) of information along the edges
of the graph. This is more clearly seen in the case of a paired delay $\phi
= \lam_{ij}$, for which the information $I_\phi = I_i + I_j + I_{ij}
1_{\{ \{i,j\} \in E\}}$ increases (hence the asymptotic delay decreases)
if there is an edge between nodes $i$ and $j$. This has no counterpart
in the classical theory where one looks at change points independently.

\medskip
\noindent \emph{Remark 3.} Another feature of the result is observed
for a single delay,
say $\phi = \lam_1$, where one has $I_\phi = I_1$ regardless of whether there are
  edges between node 1 and the rest of the nodes. Thus, the asymptotic delay for the
  threshold rule which bases its decision on the posterior
  probability of $\lam_1$ given all the data in the network ($\Xbs^n$)
  is the same as the one which bases its decision on the posterior given only
  private data of node $1$ ($\Xb_1^n$). Although this rather
  counter-intuitive result holds asymptotically, the simulations show
  that  even for moderately
  low values of $\alpha$, having access to extra information in
  $\Xbs^n$ does indeed improve performance as one
  expects. (cf. Section~\ref{sec:sims}).

  \medskip \emph{Remark 4.} The assumptions of bounded likelihood
  ratios ($\|\log \frac{f_e}{g_e} \|_\infty \le M$) and geometric
  priors on $\{\lam_j\}$ are crucial for our proof technique. The geometric distribution can be relaxed to any
  distribution with exponential tails, but we cannot allow for more
  heavy-tailed priors. A brief explanation is provided after
  stating Theorem~\ref{thm:epseq:R:S:M:L} in
  Section~\ref{sec:marg:like}. This theorem is a key ingredient in our
  argument and relies heavily on these assumptions. Exponential tails
  assumption is also used in the decoupling Lemma~\ref{lem:sum:prod}.
  

\section{Simulations}
\label{sec:sims}
\begin{figure}[!t]
\centering
\includegraphics[width=1.4in]{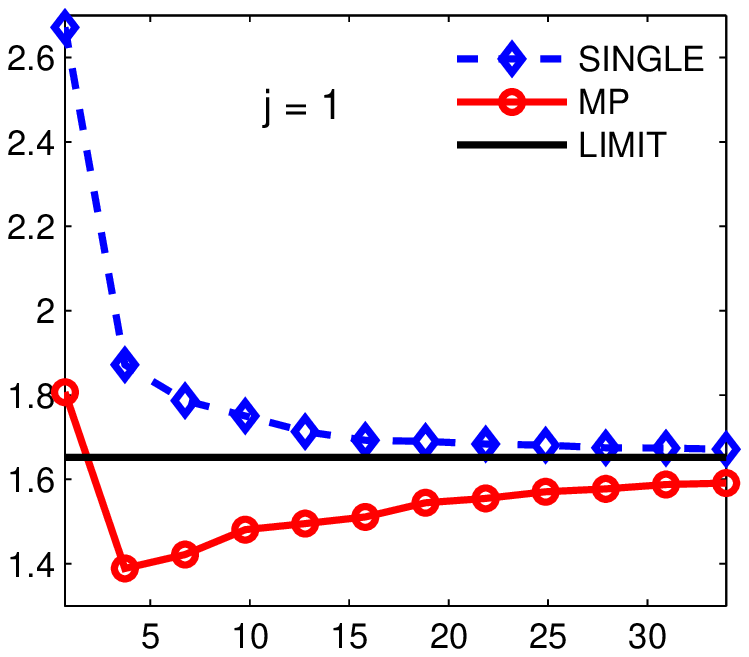}
\includegraphics[width=1.4in]{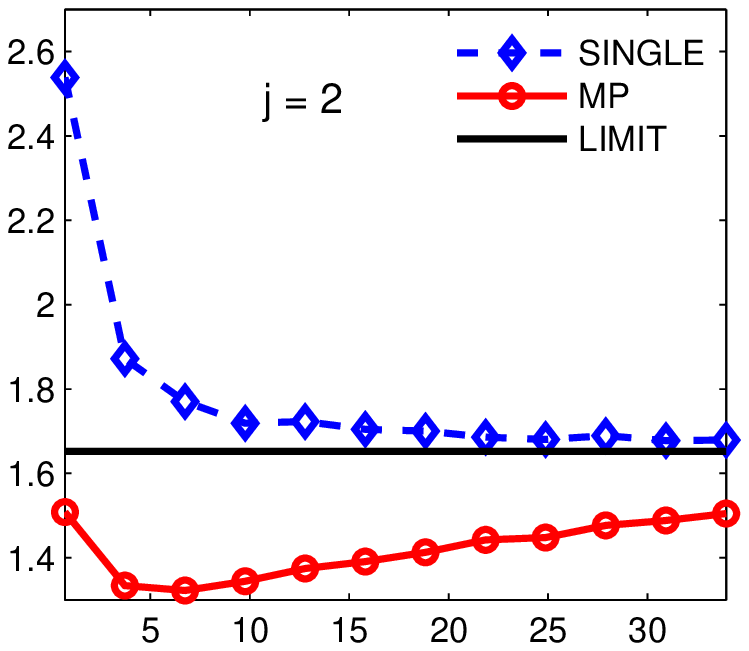}
\includegraphics[width=1.4in]{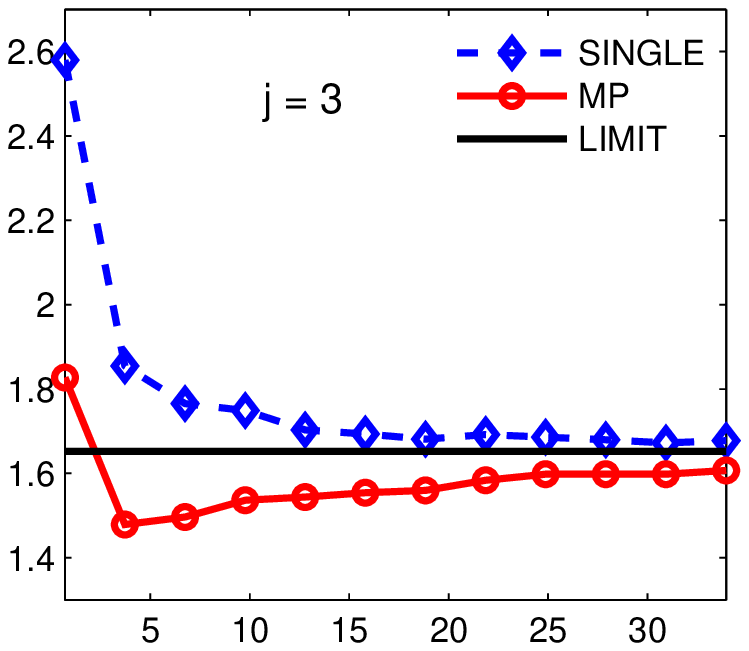}
\includegraphics[width=1.4in]{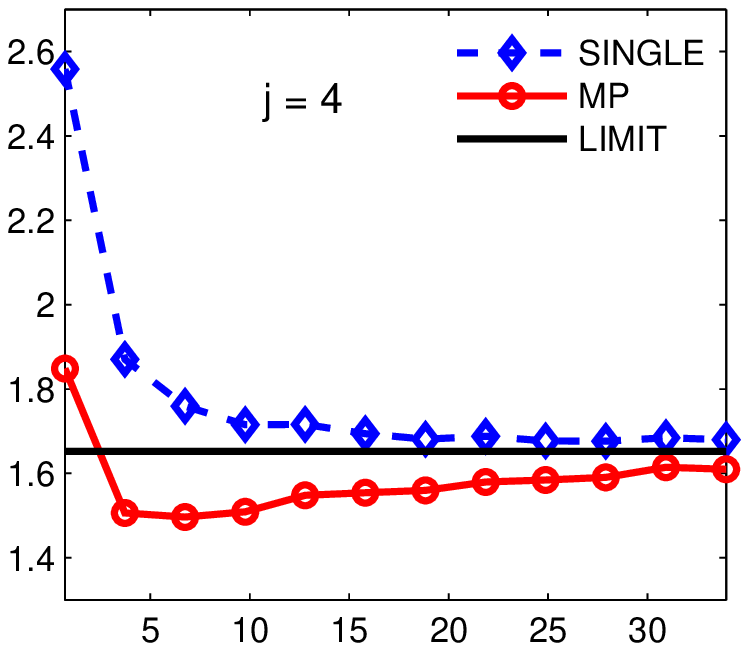}
\includegraphics[width=1.4in]{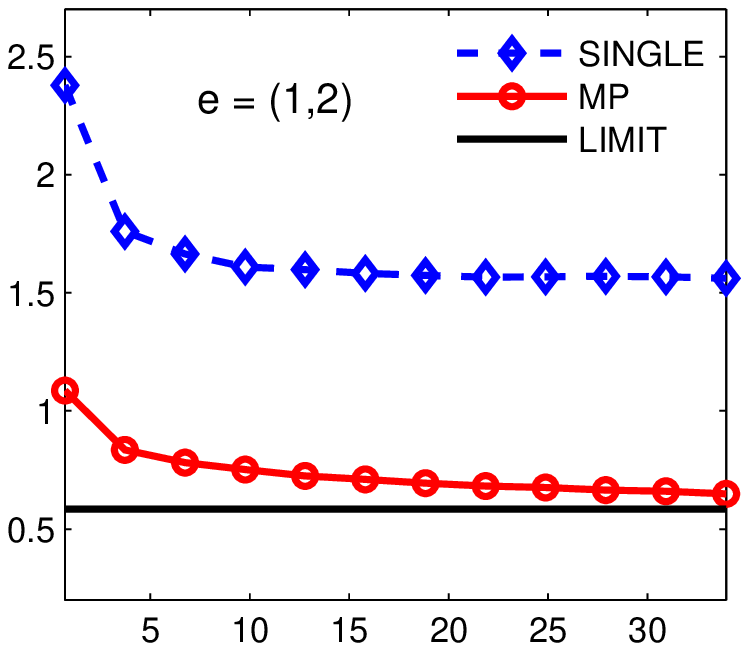}
\includegraphics[width=1.4in]{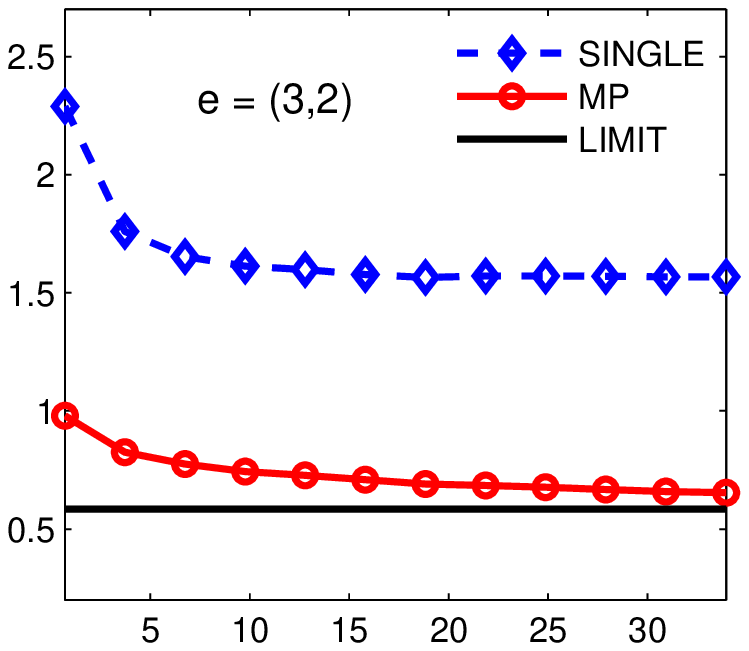}
\includegraphics[width=1.4in]{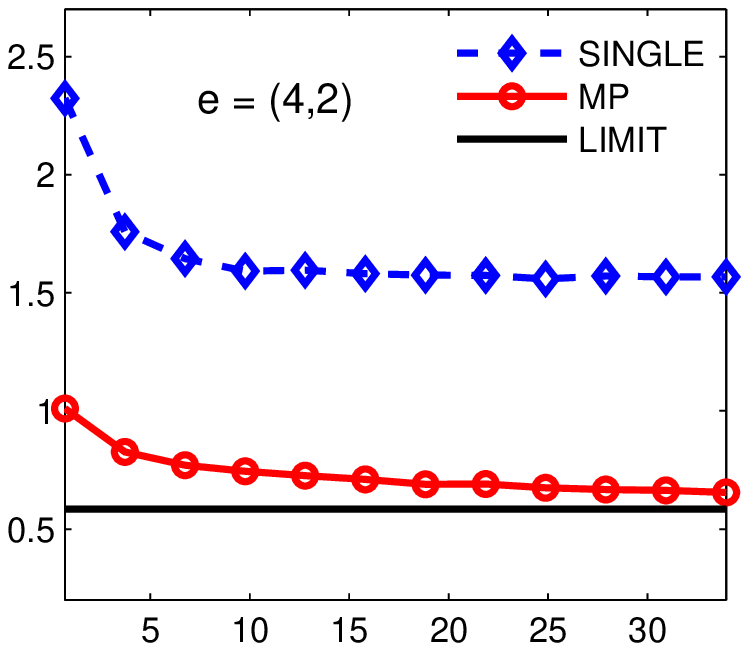}
\caption{Plots of the slope $\frac{1}{-\log \alpha}\ex[\tau_\Ss - \phi | \tau_\Ss
  \ge \phi] $ against $-\log \alpha$ for message-passing algorithm
(MP) and SINGLE algorithm which
disregards shared information. The graph is the star graph of $4$
nodes with node $2$ in the center. Estimates of both single and paired change
points ($\lambda_j$ and $\lambda_{ij}$) are shown together with
theoretical limit of Theorem~\ref{thm:asym:opt}. 
False alarm tolerance $\alpha$ ranges in $[0.5,10^{-13}]$.
}
\label{fig_sim}
\end{figure}
We present simulation results as depicted in
Fig.~\ref{fig_sim}. The setting is that of graphical
model~(\ref{eq:joint:def}) on $d = 4$ nodes, where the statistical graph is
a star with node $2$ in the middle. 
Conditioned on $\lams$, all the data sequences, $\Xbs$,
are assumed Gaussian of variance $1$, with pre-change mean $1$ and
post-change mean zero. All priors are geometric with parameters $\rho_j
= 0.1$. Fig.~\ref{fig_sim} shows plots of expected
delay over $|\log \alpha|$, against $|\log \alpha|$, for two methods:
the message-passing algorithm of Section~\ref{sec:exact:alg} (MP) and the
method which bases its inference on posteriors calculated based
only on each node's private information (SINGLE). This latter method
estimates a single change point $\lam_j$ by $\tauh_j := \inf \{ n:\, P(\lam_j
\le n | \Xb_j^n) \ge 1- \alpha \}$ and a paired $\lam_{ij} = \lam_i \wedge
\lam_j$ by $\tauh_i \wedge \tauh_j$. Also shown in the figure is the
limiting value of the normalized expected delay as predicted by
Theorem~\ref{thm:asym:opt}. All plots are generated by Monte Carlo
simulation over $5000$ realizations.

In estimating single change points, MP, which takes shared
information into account, has a clear advantage over SINGLE, for high
to relatively low false alarm values (even, say, around $\alpha
\approx e^{-5}$); though, both methods seem to converge to the same
 slope in the $\alpha \to 0$ limit, as suggested by
Theorem~\ref{thm:asym:opt}. (The particular value is $(-\log
  0.9 + 0.5)^{-1} = 1.6519$.) Also note that the advantage of MP over
SINGLE is more emphasized for node $2$, as expected by its access to
shared information from all the three nodes. 

For paired change points, the
advantage of MP over SINGLE is more emphasized. It is also
interesting to note that while MP seems to converge to the expected
theoretical limit $(-2\log0.9 + 3\cdot 0.5)^{-1} = 0.5845$, SINGLE seems to converge to a higher
slope (with a reasonable guess being $1.6519$ as in the case of single
change points). 

In regard to false alarm probability, nonzero values were only observed for the
first few values of $\alpha$ considered here, and those were either
below or very close to the specified tolerance.

\section{Concentration inequalities for marginal likelihood ratios}
\label{sec:marg:like}
In this section, we lay the groundwork for the proof of
Theorem~\ref{thm:asym:opt}. The main result here is
Theorem~\ref{thm:epseq:R:S:M:L}, which establishes concentration
inequalities for various terms that appear in an asymptotic expansion of the marginal likelihood ratio, defined in~\eqref{eq:Dphi:def} below. These terms (cf.~\eqref{eq:Sterm:def} and~\eqref{eq:Mterm:def}) are natural by-products of marginalization over a graph  and their asymptotic behavior might be of independent interest.

Our standing assumption throughout is that the graph $G=(V,E)$ is complete. This simplifies the arguments without loss of generality, since one can otherwise make the graph complete, by assigning sequences of i.i.d. data to each non-edge (with the same pre- and post- change distributions).  These i.i.d. data do not affect the likelihood (as can be verified by examining the representation of Lemma~\ref{lem:Dphi:rep}) and they do not contribute to asymptotic delay since the corresponding KL informations are zero.   

Fix some delay functional $\phi = \tau_\scal$ throughout this section.
We use the following notation regarding
conditional probabilities and expectations
\begin{align*}
  \pr_\phi^k &:= \pr(\;\cdot\; \mid \phi = k), \quad
  \ex_\phi^k := \ex(\;\cdot\; \mid \phi = k) \\
  \pr_{\lams}^{\ms} &:= \pr(\;\cdot\;\mid \lams = \ms), \quad
  \ex_{\lams}^{\ms} := \ex(\;\cdot\;\mid \lams = \ms), \;
\end{align*}
for $k \in \nats$ and $\ms = (m_1,\dots,m_d) \in \nats^d$. Here
$\{\lams = \ms\} = \cap_{j=1}^d \{ \lam_j = m_j\}$. Furthermore, let
\begin{align}\label{eq:pipk:def}
  \pipk(\ms) := \pr (\lams = \ms \mid \phi = k).
\end{align}

Consider the marginal likelihood ratio
\begin{align}\label{eq:Dphi:def}
  \Dphi^{k,n} := \Dpk(\Xbs^n) := \frac{P(\Xbs^n\mid \phi = k)}{P(\Xbs^n\mid \phi = \infty)}.
\end{align}
Our asymptotic analysis hinges on the behavior of $\frac1n \log
\Dpk(\Xbs^n)$ as $n \to \infty$, under probability measure $\pr_\phi^k$. In particular, as a  direct
consequences of the results of~\cite{TarVee05}, if one can show that
\begin{align}\label{eq:suff:cond:up}
  \pr_\phi^k \Big[ \frac1{N} \max_{1\le n \le N}\, \log \Dpk(\Xbs^{k+n}) \ge
  (1+\eps) I_\phi \Big] \stackrel{N \to \infty}{\longrightarrow} 0
\end{align}
for all (small) $\eps > 0$ and all $k \in \nats$, then the ``lower bound''
follows, $\inf_{\taut \in \rclass_\phi(\alpha)} \ex \big[ \taut - \phi \mid \taut
  \ge \phi\big] \ge \frac{| \log \alpha | }{q_\phi + I_\phi} \big( 1+
  o(1) \big)$. 
Furthermore, let 
\begin{align*}
  T_\eps^k := \sup \Big\{ n \in \nats: \; \frac{1}{n}\log
  \Dpk(\Xbs^{k+n-1}) < I_\phi - \eps \Big\}.
\end{align*}
By the results of \cite{TarVee05}, if one has
\begin{align}\label{eq:suff:cond:down}
  \ex \,T_\eps^\phi := \sum_{k=1}^\infty \pr (\phi =k) \,\ex_\phi^k(
  T_\eps^k) < \infty,
\end{align}
for all (small) $\eps > 0$, then the ``upper bound'' follows, that is,
$\tau_\Ss$ as defined in~(\ref{eq:tau:S:def}) satisfies $\ex [ \tau_\Ss - \phi \mid
\tau_\Ss \ge \phi] \le \frac{| \log \alpha|}{q_\phi +
  I_\phi}(1+o(1))$.

The following lemma provides sufficient conditions based on 
concentration inequalities under conditional probability measures
$\pr_{\lams}^{\ms}$. In the following $\eps_0>0$ is some constant. (See Appendix~\ref{app:proof:suff:cond} for the proof.)

\begin{lem}\label{lem:suff:cond}
  Assume that for all $\ms \in \nats^d$ for which $\pipk(\ms) > 0$,
  one has
  \begin{align}
    \pr_{\lams}^{\ms} \Big\{ \Big| \frac1n \log \Dpk(\Xbs^n) -
    I_\phi\Big| > \eps\Big\} \le q(n) \exp(-c_1 n \eps^2)
  \end{align}
  for  all $n\in \nats$ and $\eps \in (0,\eps_0)$ such that $\sqrt{n} \ge
  \frac{1}{\eps} p(\ms,k)$, where $p(\cdot)$ and $q(\cdot)$ are
  polynomials with constant nonnegative coefficients. Furthermore, assume that
  both $\pipk(\cdot)$ and $\pr(\phi = \cdot)$ have finite polynomial
  moments. Then
  both~(\ref{eq:suff:cond:up}) and~(\ref{eq:suff:cond:down}) hold,
  hence Theorem~\ref{thm:asym:opt} holds.
\end{lem}

\noindent \emph{Remark 1.} The condition of finite polynomial
moments for $\pipk(\cdot)$ and $\pr(\phi = \cdot)$ is satisfied for
a $\phi = \min_{j \in \Ss} \lam_j$ under
geometric priors on $\{\lam_j\}$.

\medskip
In order to apply Lemma~\ref{lem:suff:cond} easily, we introduce a notion of
``stochastic asymptotic $\eps$-equivalence'' for sequence of random
variables. To simplify notation, let $\supp(\pipk) := \{ \ms \in
\nats^d :\; \pipk(\ms) > 0 \}$.

\begin{defn}\label{def:epseq}
  Consider two sequences $\{a_n\}$ and $\{b_n\}$ of random variables, where
  $a_n = a_n(k)$ and $b_n = b_n(k)$ could depend on a common
  parameter $k \in \nats$. The two sequences are called
  ``asymptotically $\eps$-equivalent'' as $n \to \infty$, w.r.t. the collection
  $\{\pr_{\lams}^{\ms}:\; \ms \in \supp(\pipk)\}$, and denoted 
  \begin{align*}
   a_n \epseq
  b_n,
  \end{align*} if there exist polynomials $p(\cdot)$ and $q(\cdot)$ (with constant
  nonnegative coefficients), and $\eps_0 > 0$, such that for all $\ms
  \in \supp (\pipk)$, we have
\begin{align*}
  \pr_{\lams}^{\ms} (|a_n - b_n| \le \eps ) \ge 1 - q(n) e^{-c_1 n
  \eps^2}
\end{align*}
 for all $n \in \nats$ and
$\eps \in (0,\eps_0)$ satisfying $\sqrt{n} \eps \ge 
p(\ms,k)$.  The one-sided version, e.g,  $a_n \epsle
b_n$ is defined by replacing $|a_n - b_n| \le \eps$ with $a_n \le b_n
+ \eps$. (The constants are independent of $n,\ms, k$, and $\eps$, but they
could depend on other parameters of the problem.)
\end{defn}


By application of union bound and algebra, a finite number of asymptotic
$\eps$-equivalence statements can be manipulated under some algebraic
rules to produce new such statements. Below, we summarize some of the
rules:
\begin{enumerate}[(R1)]
  \item $a_n \epseq b_n$ implies $a_n \asyma{C \eps} b_n$ for $C >
    0$ and $\alpha a_n \epseq \alpha b_n$ for $\alpha \in
    \mathbb{R}$.
  \item $a_n \epseq b_n$ and $b_n \epseq c_n$
    implies $a_n \epseq c_n$. (Transitivity)
  \item \label{itm:sum}
  $a_n \epseq b_n$ and $c_n
  \epseq d_n$ implies $a_n \pm c_n \epseq b_n \pm d_n$.
  \item \label{itm:max}
  $a_n \epseq b_n$ implies $\max \{ a_n, c_n\} \epseq \max
    \{ b_n, c_n\} $.
  \item \label{itm:scale:one} 
  $a_n \epseq b_n$, $c_n \epseq 1$ and $\{b_n\}$ bounded implies $a_n c_n \epseq b_n$.
  \item \label{itm:eq:ineq}
  $a_n \epseq a > 0$ and $b_n \epsle -b < 0$ implies $\max\{a_n,b_n\} \epseq a$.
  \item \label{itm:log:sum:max} 
   ``log--sum-max'' inequality for positive sequences $\{a_n\}$
    and $\{b_n\}$:
    \begin{align}\label{ineq:log:sum:max}
      n^{-1} \log( a_n + b_n) \epseq \max \{ n^{-1}\log a_n,
  n^{-1} \log b_n\}.
    \end{align}
\end{enumerate}
The last statement follows from inequalities $0 \le \log(a_n + b_n) -
\max\{ \log a_n, \log b_n\} \le
\log 2$. Dividing by $n$, we observe that the difference is bounded by
$\eps$, in absolute value, as long
as $ n \eps \ge \log 2$. This  implies the condition in Definition~1,
since  $\{ (n,\eps) : \sqrt{n} \eps \ge \log 2\} \subset \{ (n,\eps) : n
\eps \ge \log 2\}$.

As another example of how these rules are obtained, consider
(R\ref{itm:sum}). We have $|a_n - b_n| \le \eps$ on event $A_{1,n}$
having probability at least $1-q_1(n) e^{-c_1 n\eps^2}$, for $\sqrt{n}
\eps \ge p_1(\ms,k)$. Similarly, $|b_n -c_n| \le \eps$ on event
$A_{2,n}$ with probability at least $1 - q_2(n) e^{-c_2 n \eps^2}$,
for $\sqrt{n} \eps \ge p_2(\ms,k)$. Then, by union bound $A_{1,n} \cap
A_{2,n}$ has probability at least $1 - (q_1(n) + q_2(n)) e^{-(c_1
  \wedge c_2) n \eps^2}$, for $\sqrt{n} \eps \ge p_1(\ms,k) +
p_2(\ms,k)$. For this range of $n$, on event $A_{1,n} \cap A_{2,n}$,
we have both $|a_n -b_n|\le \eps$ and $|b_n - c_n| \le \eps$, from
which it follows $|a_n - c_n| \le 2\eps$, by triangle inequality. Since
both $q_1 + q_2$ and $p_1 + p_2$ are polynomials, we have the desired
assertion.

\medskip
\noindent \emph{Remark 1}
According to Definition~\ref{def:epseq} and Lemma~\ref{lem:suff:cond}, to prove Theorem~\ref{thm:asym:opt}, it is enough to show that 
\begin{align*}
\frac{1}{n} \log \Dphi^{k,n} \epseq I_\phi
\quad \text{as $n \to \infty$ w.r.t.
$\{\pr_{\lams}^{\ms}\}$}
\end{align*}
(We often omit $\ms \in \supp(\pipk)$ when it is implicitly understood.) The rules stated above allows one to reduce the problem to asymptotic $\eps$-equivalence statements for simpler terms, as considered in the next section. In this context, we regard parameters of the priors, $\{\rho_j\}$, and pre- and post-change densities as constants. In other words, the constants in the definition of $\eps$-equivalence can depend on $\{\rho_j\}$, $\{I_e\}$, and $M$ (the uniform norm of $\log (f_e/g_e)$).


\medskip
We now introduce a couple of building blocks occurring frequently and establish $\epseq$ statements for them. Recall that $f_e$ and $g_e$ denote the pre- and post-change densities for edge $e \in \Et$. Define
\begin{align}\label{eq:R:def}
	R_k^n(e) := R_k^n(\Xb_e) := 
		\prod_{t=k}^{n} \frac{f_e}{g_e}(\Xb_e) =
		 \prod_{t=k}^{n} e^{h_e(\Xb_e)}, \quad 
	h_e := \log \frac{f_e}{g_e}.
\end{align}
Note that by assumption $\|h_e\|_\infty \le M$ for all $e$. We will use the convention that empty products evaluate to $1$, that is, $R_k^n(e) = 1$ whenever $k > n$. We also define $\Sterm$-terms as
\begin{align}\label{eq:Sterm:def}
	\Sterm_{\mui}^{\nui,n}(e) := 
		\sum_{p=\mui}^\nui A e^{-\beta p} R_p^n(e)
\end{align}
where $A$ and $\beta$ are some positive constants. Similarly, define $\Mterm$ and $\Lterm$-terms as follows
\begin{align}\label{eq:Mterm:def}
	\Mterm_{\mui}^{\nui,n}(e) &:= 
	\sum_{p_1=\mui}^{\nui} \sum_{p_2=\mui}^\nui
		A e^{-(\beta_1 p_1 + \beta_2 p_2)} 
		R_{p_1 \wedge p_2}^n (e)  \\
	\Lterm_{\mui,(\rr)}^{\nui,n}(e) &:= 
	\sum_{p=\mui}^{\nui} A e^{-\beta p} 
		R_{p \wedge \rr}^n(e) \label{eq:Lterm:def}
\end{align}
for constants $A,\beta_1,\beta_2,\beta > 0$. The constants involved in these definitions can be different in each occurrence and we have suppressed them in the notation for simplicity. The $\Mterm$ and $\Lterm$-terms are most relevant when $e$ is a proper edge, that is, $e = \{i,j\} \in \Et$ and $i \neq j$, although the statements involving them hold in general. 

The following lemma is proved in Section~\ref{sec:epseq:proofs}. Recall that $I_e$ is the KL divergence between $f_e$ and $g_e$, that is, $I_e := \int f_e \log \frac{f_e}{g_e}$.

\begin{thm}\label{thm:epseq:R:S:M:L}
Assume $ \|\log \frac{f_e}{g_e}\|_\infty \le M$ for all $e
  \in \Et$. The following asymptotic $\eps$-equivalence relations hold with respect to $\{\pr_{\lams}^{\ms}: \; \ms \in \supp(\pipk)\}$, as $n \to \infty$,
\begin{align}
	\frac{1}{n} \log R_\mui^n(e) \,\epseq\, 
	\frac{1}{n} \log S_{\mui}^{\infty,n}(e) \,\epseq\,
	\frac{1}{n} \log M_{\mui}^{\infty,n}(e) \,\epseq\,
	\frac{1}{n} \log L_{\mui,(r)}^{\infty,n}(e) \,\epseq\, 
	I_e
\end{align}
for any  $\mui,r \le 2k$ and $e \in \Et$.
\end{thm}
The proof of this theorem is deferred to Section~\ref{sec:epseq:proofs}.
The $\log \eps$-equivalence $\frac{1}{n} \log R_\mui^n(e) \epseq I_e$
is intuitive as will become clear in the proof. The lemma essentially
states that there are no surprises regarding $\Sterm$, $\Mterm$ and
$\Lterm$ terms and they are all $\eps$-equivalent to the corresponding
edge information. 
We also note that $2k$ in the statement of the Lemma can be replaced
with $Ck$ for any constant $C > 0$. 

\medskip
\noindent \emph{Remark.}
Let us consider the role of our assumptions on the priors and likelihood ratios, by giving a high-level overview of the proof of Theorem\ref{thm:epseq:R:S:M:L} for $\Sterm_\mui^{\infty,n}(e)$. The exponential decay for the tails of the priors is reflected in the definition of $\Sterm_\mui^{\infty,n}(e)$ in in~(\ref{eq:Sterm:def}). The terms $R_p^n(e)$ in this sum are concentrated around $I_e$ if $p \ll n$ (as in this case $R_p^n(e)$ is the product of many essentially i.i.d. terms). For $p$ close to $n$, however, there is no guaranteed concentration for $R_p^n(e)$, as it is a product of only a few random variables. For  these terms, however, the prefactor $e^{-\beta p}$ is small while $R_p^n(e)$ is gauranteed to be bounded (based on $ \|\log (f_e/g_e)\|_\infty \le M$). Hence these terms are a negligible and do not contribute to $\frac1n \log \Sterm_\mui^{\infty,n}(e)$, asymptotically. This argument is made precise in Section~\ref{sec:epseq:proofs}.

\bigskip
To simplify notation, from now on, we will drop the second upper index in the symbols for $\Sterm$, $\Lterm$ and $\Mterm$ terms, whenever this index is $n$ and there is no chance of confusion.  That is, we adhere to the following convention,
\begin{align}\label{eq:short:notation}
    \Sterm_\mui^\nu(e) := \Sterm_\mui^{\nu,n}(e), \quad
    \Lterm_{\mui,(r)}^\nu(e) := \Lterm_{\mui,(r)}^{\nu,n}(e), \quad
    \Mterm_\mui^\nu(e) := \Mterm_\mui^{\nu,n}(e).
\end{align}

  

\section{Proof of the optimal delay theorem}
Let us define
\begin{align}\label{eq:mixR:def}
  \mixR_\phi^{k,n} := \mixR_\phi^k(\Xbs^n) := 
  \sum_{k_1,\dots,k_d}\pipk(k_1,\dots,k_d)
  \prod_{j \in V} R_{k_j}^n\{j\} \prod_{e \in E} R_{k_e}^n(e)
\end{align}
where $k_e := k_i \wedge k_j$ for $e = \{i,j\}$, and each variable
$k_j$ runs over $\{1,2,\dots\} \cup \{ \infty\}$. The inclusion of $\infty$ in range of the summations does not affect the case $k < \infty$, but will allow us to use the same expression~\eqref{eq:mixR:def} for $\mixR_\phi^{\infty,n}$.
We have following easily verified representation of $\Dphi^{k,n}$. (See Appendix~\ref{app:Dphi:rep:proof} for the proof).

\begin{lem}\label{lem:Dphi:rep}
  With $\Dphi^{k,n}$ defined as in~\eqref{eq:Dphi:def},  
  \begin{align}\label{eq:Dpk:gen}
    \Dphi^{k,n} = \frac{\mixR_\phi^{k,n}}{\mixR_\phi^{\infty,n}}.
  \end{align}
\end{lem}

We will use the following technical lemma to decouple sums of
products.  Let $\binom{[r]}{2}$ denote the collection of 2-subsets of
$[r] = \{1,\dots,r\}$, with the convention that each member is a
denoted as an ordered pair $(i,j)$ with $i < j$.
\begin{lem}
\label{lem:sum:prod}
  Let $\Sb = S_1 \times S_2 \times \cdots \times S_r$ be the Cartesian product of $r$ countable sets $S_1,\dots,S_r$ and let $\kb = (k_1,\dots,k_r)$ be a multi-index for $\Sb$. Let $F_j$ and $G_{ij}$ be nonnegative functions defined on $S_j$ and $S_i
  \times S_j$ respectively, for $i,j \in [r]$. Let $H_1$ be a nonnegative function on $S_1$. Let $\betab =
  (\beta_1,\dots,\beta_r) \in \realspp^r$. Then,
  \begin{align}\label{eq:sum:prod:2}
    (a) &\quad \sum_{k_1 \in S_1} e^{- \beta_1 k_1 } F_1(k_1) H_1(k_1) \le
    \Big( \sum_{k_1 \in S_1} e^{-\beta_1 k_1/2} F_1(k_1) \Big)
    \Big( \sum_{k_1 \in S_1} e^{-\beta_1 k_1/2} H_1(k_1) \Big).
  \\
  \label{eq:sum:prod:1}
  \begin{split}
    (b) &\quad \sum_{\kb\, \in\, \Sb} \Big\{ 
    e^{- \betab^T \kb} \prod_{j=1}^r F_j(k_j) \prod_{(i,j) \in
      \binom{[r]}{2}}
    G_{ij}(k_i,k_j)
    \Big\} \;\le\;
     \Big( \prod_{j=1}^r \Big\{ \sum_{k_j \in S_j}e^{-\beta_j k_j/r}
     F_j(k_j) \Big\} 
     \Big) \times \\
     & \qquad \qquad \qquad \qquad \qquad \prod_{(i,j) \in \binom{[r]}{2}} \Big\{ \sum_{(k_i,k_j) \in S_i
       \times S_j}e^{-(\beta_i k_i +
       \beta_j k_j)/r}
     G_{ij}(k_i,k_j)\Big\}
  \end{split}
  \end{align}
\end{lem}
The key in this lemma is that the functions $F_j$, $G_{ij}$ and $H_1$ are nonnegative.
One might already see how the application of Lemma~\ref{lem:sum:prod}
to the sum in~\eqref{eq:mixR:def} produces $\Sterm$ and $\Mterm$ terms
as introduced in Section~\ref{sec:marg:like}. 
We are ready to give
the proof of Theorem~\ref{thm:asym:opt}. We start with the two extreme change point functionals
$\lam_\Ss$: a single change point ($|\Ss|= 1$), and the minimum of all
the change points ($\Ss = [d]$). Then, we present the proof for
$\lam_\Ss$ with $1 < |\Ss| < d$, omitting some of the details for brevity.

\subsection{Proof for the case $\phi = \lam_1 \wedge \cdots \wedge \lam_d$}
\label{sec:proof:min:all}
First, note that in this case $\mixR_\phi^{\infty, n} = 1$, since
$\phi = \infty$ implies $\lam_1 = \cdots = \lam_d = \infty$. Hence, we
only need to consider $\mixR_\phi^{k,n}$ for some $k < \infty$. We
then observe that $\pipk(k_1,\dots,k_d)$ is nonzero only when at least
one of $k_1,\dots,k_d$ is equal to $k$. We break up the sum
according to how many of $k_1,\dots,k_d$ are equal to $k$.

Let $\ical$ be a subset of $[d]$ of size $|\ical| = s$. Let $\ical^c =
[d] \setminus \ical$. Consider the terms in the
sum~(\ref{eq:mixR:def}) for which $k_j = k$ for $j \in \ical$ and $k_j
> k$ for $j \in \ical$. We call the sum over these terms
$\Tterm_\ical$. Then, $\mixR_\phi^{k,n} = \sum_{\ical : |\ical| \ge 1}
\Tterm_\ical = \sum_{s = 1}^d \sum_{\ical : |\ical| = s}
\Tterm_\ical$, where the sum is over all subsets $\ical$ of $[d]$ of
size at least $1$.

Let us fixed some $s \in [d]$ and some $\ical \subset [d]$ with $|\ical|
= s$. Without loss of generality, we can pick $\ical =
\{1,\dots,s\}$. We note that
\begin{align*}
  \pipk(k,\dots,k,k_{s+1},\dots,k_d) &= A \prod_{j = s+1}^d 
  \rhob_j^{k_j} \\
  &= A e^{-\sum_{j \in \ical^c} \beta_j k_j}, \quad \text{for}\; k_j >
  k, j \in \ical^c
\end{align*}
where $\beta_j = - \log \rhob_j > 0$, and $A = A(\{\rho_j\})$ is some
constant. It follows that
\begin{align} \label{eq:Tj:expan:1}
  \Tterm_\ical = \prod_{j \in \,\ical} R_k^n\{j\}  \prod_{|e \,\cap\,
    \ical|\, \ge\, 1} R_k^n(e)
  \underbrace{\sum_{k_j > k, \; j \in \ical^c} \Big\{A e^{ -\sum_{j \in \ical^c} \beta_j
    k_j}
  \prod_{j \in \,\ical^c} R_{k_j}^n\{j\} \prod_{|e \,\cap\, \ical| = 0}
  R_{k_{e}}^n(e) \Big\}}_{(\star)}.
\end{align}
Here and in the rest of the proof, the index $e$ runs in the set $E$
of original edges (not the modified set $\Et$ introduced in Section~\ref{sec:optimal:delay:thm}). That is,
each edge $e = \{i,j\} \in E$ for some $i \neq j$. 
Note that in~(\ref{eq:Tj:expan:1}), the rightmost product is over all $2$-subset of
$\ical^c$, which we denote as $\binom{\ical^c}{2}$. We can now apply first part of Lemma~\ref{lem:sum:prod},
with $r = |\ical^c| = d -s$, to obtain
\begin{align*}
  (\star) \le A \prod_{j \in \ical^{c}} \Big\{ \underbrace{\sum_{k_j > k} e^{-\frac{\beta_j k_j}{ d-s}}
  R_{k_j}^n\{j\}}_{(\star\star)} \Big\}
 \prod_{(i,j) \in \binom{\ical^c}{2}} \Big\{  \underbrace{
  \sum_{\substack{k_i > k, \\k_j > k}} e^{-\frac{\beta_i k_i + \beta_j
    k_j}{d-s}} R_{k_i \wedge k_j}^n\{i,j\} }_{(\star\star\star)} \Big\}.
\end{align*}
Each term denoted as $(\star\star)$ is of the form
$\Sterm_{k+1}^{\infty\cn}\{j\}$ and each term denoted as
$(\star\!\star\!\star)$ is of the form
$\Mterm_{k+1}^{\infty\cn}\{i,j\}$. Hence, we have
\begin{align*}
  \Tterm_\ical \; \le A \prod_{j \in \,\ical} R_k^n\{j\}  \prod_{|e \,\cap\,
    \ical|\, \ge\, 1} R_k^n(e) \prod_{j \in \ical^c}
  \Sterm_{k+1}^{\infty\cn}\{j\}
  \prod_{|e \,\cap\, \ical| = 0} M_{k+1}^{\infty\cn}(e).
\end{align*}
Applying Theorem~\ref{thm:epseq:R:S:M:L} to each of the $R, \Sterm$ and $\Mterm$ forms above,
we obtain
\begin{align}
  \frac1n \log \Tterm_\ical &\epsle \sum_{j \in \ical} I_j + \sum_{|e
    \cap \ical| \ge 1} I_e + \sum_{j \in \ical^c} I_j +
  \sum_{|e\,\cap\,\ical| = 0} I_e \notag \\
  &= \sum_{j \,\in\, V} I_j + \sum_{e \,\in\, E} I_e.
  \label{eq:TI:upper:bnd}
\end{align}
where the $\eps$-equivalence in the above an in what follows is w.r.t. $\{\pr_{\lams}^{\ms}\}$.

To obtain the lower bound, we bound $(\star)$ from below by its first
term,
\begin{align*}
  (\star) \ge A e^{-(k+1) \sum_{j \in \ical^c} \beta_j} 
  \prod_{j \,\in\, \ical^c} R_{k+1}^n\{j\} \prod_{|e\,\cap\,\ical| =
    0} R_{k+1}^n(e)
\end{align*}
which, after applying Theorem~\ref{thm:epseq:R:S:M:L},  gives us a lower bound on $\frac1n \log \Tterm_\ical$ matching
the RHS of~(\ref{eq:TI:upper:bnd}). Finally, note that the RHS
of~(\ref{eq:TI:upper:bnd}) does not depend on the particular choice of
$\ical$. We now use the log-sum-max rule~\Rule{\ref{itm:log:sum:max}} to get
\begin{align*}
  \frac1n \log \mixR_\phi^{k,n} &= \frac{1}{n} \log \Big(
  \sum_{\ical: |\ical|\ge 1} \Tterm_\ical \Big) \\
  &\epseq \max_{\ical: |\ical|\ge 1} \Big\{ \frac1n \log
  \Tterm_\ical\Big\} \epseq \sum_{j \,\in\, V} I_j + \sum_{e \,\in\, E} I_e,
\end{align*}
which is the desired result.

\subsection{Proof for the case $\phi = \lam_1$}
\label{sec:proof:lam1}
In this case, one has $\pipk(k_1,\dots,k_d) = 1\{k_1 = k\}
\prod_{j=2}^d \pi_j(k_j)$, hence
\begin{align*}
  \pipk(k_1,\dots,k_d) 
  =  \Big[ 1\{ k_1 = k\} A \prod_{j=2}^d \rhob_j^{k_j} \Big] =
  \Big[ 1\{k_1 = k\} A e^{-\sum_{j=2}^d \beta_j k_j} \Big]
\end{align*}
where $\beta_j = -\log \rhob_j > 0$ and $A = A(\{\rho_j\})> 0$ is some
constant. Then, we can write 
\begin{align}\label{eq:proof:lam1:temp}
  \mixR_\phi^{k,n} = R_{k}^n\{1\} \sum_{k_2,\dots,k_d} \Big\{ A e^{-
    \sum_{j=2}^d \beta_j k_j} \prod_{j=2}^d \Big(R_{k_j}^n\{j\}  R_{k
    \wedge k_j}^n \{1,j\}\Big) \, \prod_{e:\; 1 \notin e} R_{k_e}^n(e) \Big\}.
\end{align}
Note that the second product runs over all $2$-subsets of
$[2:d]:= \{2,\dots,d\}$ which we denote as $\binom{[2:d]}{2}$. Hence, we can apply Lemma~\ref{lem:sum:prod} with $r
= d-1$ to obtain
\begin{align*}
  \frac{\mixR_\phi^{k,n}}{ A R_{k}^n\{1\}} \le  
   \prod_{j=2}^d \Big\{ \sum_{k_j} e^{-\beta_j k_j /(d-1)}R_{k_j}^n\{j\}  R_{k
    \wedge k_j}^n \{1,j\} \Big\} \prod_{(i,j) \in \binom{[2:d]}{2}}
  \Big\{ 
  \sum_{k_i,k_j} e^{- (\beta_i k_i + \beta_j k_j)/(d-1)} R_{k_i \wedge k_j}^n\{i,j\}
   \Big\}.
\end{align*}
Each term appearing in second product is of the form
$\Mterm_1^{\infty\cn}\{i,j\}$. Applying the second half of
Lemma~\ref{lem:sum:prod} to the first product, we get
\begin{align}\label{eq:*:**}
  \frac{\mixR_\phi^{k,n}}{A R_{k}^n\{1\}} \le \prod_{j=2}^d \Big\{
  \underbrace{\Big( \sum_{k_j} e^{-\frac{\beta_j k_j}{2(d-1)}}
    R_{k_j}^n\{j\} \Big)}_{(*)} \underbrace{\Big( \sum_{k_j}
    e^{-\frac{\beta_j k_j}{2(d-1)}} R_{k \wedge k_j}^n\{1,j\} \Big)}_{(**)}
    \Big\} \prod_{(i,j) \in \binom{[2:d]}{2}} M_1^{\infty\cn}\{i,j\}
\end{align}
Each term denoted as $(*)$ is of the form
$\Sterm_1^{\infty\cn}\{j\}$. For $k < \infty$, each term denoted as
$(**)$ can be written in the form $\Lterm_{1,(k)}^{\infty\cn}\{1,j\}$
. That is,
\begin{align*}
  \frac{\mixR_\phi^{k,n}}{A R_{k}^n\{1\}} \le \prod_{j=2}^d \Big\{
  \Sterm_1^{\infty\cn}\{j\} \Lterm_{1,(k)}^{\infty\cn}\{1,j\}  \Big\}
   \prod_{(i,j) \in \binom{[2:d]}{2}} M_1^{\infty\cn}\{i,j\}
\end{align*}
Applying Theorem~\ref{thm:epseq:R:S:M:L} to each of the $R$,
$\Sterm$, $\Lterm$ and $\Mterm$ forms above, we obtain
\begin{align}\label{eq:lam1:mixR:up:bnd}
  \frac1n \log \mixR_\phi^{k,n} \epsle I_1 + \sum_{j=2}^d ( I_j +
  I_{1,j}) + \sum_{(i,j) \in \binom{[2:d]}{2}} I_{ij}
\end{align}
where the $\eps$-equivalence in the above an in what follows is w.r.t. $\{\pr_{\lams}^{\ms}\}$.

The lower bound is obtained, as in Section~\ref{sec:proof:min:all},  by bounding the sum in~\eqref{eq:proof:lam1:temp} by
its first term (i.e., $k_1 = k_2 = \cdots = k_d = 1$)
\begin{align*}
  \mixR_\phi^{k,n} \ge R_k^n\{1\} A e^{-\sum_{j=2}^d \beta_j}
  \prod_{j=2}^d R_{1}^n\{j\} R_1^n\{1,j\} \prod_{(i,j) \in \binom{[2:d]}{2}} R_{1}^n\{i,j\}.
\end{align*}
Applying $\frac1n \log(\cdot)$ and using
Theorem~\ref{thm:epseq:R:S:M:L} for each term, we get
a lower bound matching the RHS of~(\ref{eq:lam1:mixR:up:bnd}). That is, the bound in~(\ref{eq:lam1:mixR:up:bnd}) holds
with $\epsle$ replaced with $\epseq$.

Now consider the denominator of $\llr_\phi^{k,n}$, namely $\mixR_\phi^{\infty,n}$.
An upper bound on
$\mixR_\phi^{\infty,n}$ can be obtained by letting $k = \infty$ in
(\ref{eq:*:**}). We note that $R_\infty^n\{1\} = 1$ and that $(**)$ is
now a term of the form $\Sterm_1^{\infty\cn}\{1,j\}$. Proceeding as before, we
obtain an upper bound similar to that of~\eqref{eq:lam1:mixR:up:bnd}, with $I_1$ missing from the bound. The lower bound is obtained by the same technique. Hence,
\begin{align}\label{eq:lam1:mixR:inf:final:}
  \frac1n \log \mixR_\phi^{\infty,n} \; \epseq\; \sum_{j=2}^d ( I_j +
  I_{1,j}) + \sum_{(i,j) \in \binom{[2:d]}{2}} I_{ij}.
\end{align}
Combining equality form of~(\ref{eq:lam1:mixR:up:bnd})
and~(\ref{eq:lam1:mixR:inf:final:}), we have
\begin{align}
  \frac1n \log \llr_\phi^{k,n} = \frac1n \log \mixR_\phi^{k,n} -
  \frac1n \log \mixR_\phi^{\infty,n} \;\epseq\; I_1
\end{align}
which is the desired result.

\subsection{Proof for $\lam_\Ss$ with $1 < |\Ss| < d$}
We now briefly give the proof for the remaining cases. Without loss of generality, we assume $\Ss = \{1,2,\dots,r\}$ for some $r \in \{2,\dots,d-1\}$. In other words, the delay functional is $\phi = \lam_\Ss= \lam_1 \wedge \lam_2 \wedge \cdots \wedge \lam_r$. We observe that $\pipk(k_1,\dots,k_d)$ is nonzero when all of $k_1,\dots,k_r$ are $\ge k$, while at least one of them is equal to $k$. Consider $\mixR_\phi^{k,n}$ for $k < \infty$. As in Section~\ref{sec:proof:min:all}, we break up the sum in its definition according to how many of $k_1,\dots,k_d$ are equal to $k$.

Let $\ical$ be a subset of $\Ss = [r]$ of size $|\ical| = s \le r$. Let $\lcal :=
\Ss \setminus \ical$ and $\Ss^c := [d] \setminus \Ss$. Note that
$\{\ical,\lcal,\Ss^c\}$ form a partition of the index set $[d]$. To
simplify notation, let $\ical^c = [d] \setminus \ical$ and note that
$\ical^c = \lcal \cup \Ss^c$.

Consider the terms in the
sum~(\ref{eq:mixR:def}) for which $k_j = k$ for $j \in \ical$ and $k_j
> k$ for $j \in \lcal$. We call the sum over these terms
$\Tterm_\ical$. Then, $\mixR_\phi^{k,n} = \sum_{s = 1}^r \sum_{\ical : |\ical| = s}
\Tterm_\ical$.

Now fix some $s \in [r]$ and some $\ical \subset \Ss$ with $|\ical| =
s$. The $R$-terms in the expression of $\Tterm_\ical$ corresponding to
nodes are easy to deal with. For the $R$-terms corresponding to edges,
we first break them into three categories, based on how many of the
endpoints are in $\ical$ (i.e., $|e \cap \ical| = 0,1,2$). The case
where exactly one endpoint is in $\ical$ (i.e., $|e \cap \ical| = 1$)
is further broken into two cases based on whether the other endpoint
is in $\lcal$ or in $\Ss^c$. The former case, i.e. $|e \cap \ical| = |e
\cap \lcal| = 1$ behaves the same as the case $|e \cap \ical| = 2$. We
thus combine these two cases, denoted as $|e \cap \ical| \ge 1, e
\subset \Ss$. To summarize, we break the edges into a total of
three categories.
We get the following decomposition
\begin{multline}\label{eq:Tj:expan:long}
  \Tterm_\ical = \prod_{j \in \,\ical} R_k^n\{j\}  \!
  \prod_{\substack{|e \,\cap\, \ical| \,\ge\, 1, \\ e \,\subset\, \Ss}}  R_{k}^n(e) 
  \;\times \\
  \sum_{ \substack{ k_j \,> \,k, \; j \in \lcal \\ k_j \,\ge\, 1, \; j
      \in \Ss^c}} \Big\{A e^{ -\sum_{j \in \ical^c} \beta_j
    k_j}
  \underbrace{\prod_{j \in \,\ical^c} R_{k_j}^n\{j\}}_{(*)}
  \underbrace{ \prod_{\substack{|e \,\cap\, \ical| = 1, \\e \,\cap\, \Ss^c =
      \{\ell\}}}   R_{k \wedge k_\ell}^n(e) }_{(**)}
  \underbrace{\prod_{|e \,\cap\, \ical| = 0} R_{k_e}^n(e)}_{(***)} \Big\}
\end{multline}
As in Sections~\ref{sec:proof:min:all} and~\ref{sec:proof:lam1}, we can apply Lemma~\ref{lem:sum:prod} to
decouple the sum and obtain an upper bound on $\Tterm_\ical$. The
products denoted by $(*)$, $(**)$ and $(*\!*\!*)$ produce $\Sterm$,
$\Lterm$, and $\Mterm$-terms\footnote{Strictly speaking, some of the
  terms produced by $(*\!*\!*)$ will have the form of an $\Mterm$-term
in the extended sense to be introduced in~(\ref{eq:extnd:Mterm:def}). For
example, we will have $\Mterm$-terms of the from
$\Mterm_{1,k+1}^{\infty,
\infty}(e)$. Since every term of the sum is nonnegative, we have the
inequality $\Mterm_{k+1}^\infty(e) \le \Mterm_{1,k+1}^{\infty,\infty}(e) \le
\Mterm_1^\infty(e)$, which in view of Theorem~\ref{thm:epseq:R:S:M:L}
implies $\frac1n \log  \Mterm_{1,k+1}^{\infty,\infty}(e) \epseq I_e$.}, respectively. Using the same
lower bounding technique and applying Theorem~\ref{thm:epseq:R:S:M:L},
we obtain
\begin{align*}
 \frac1n \log \Tterm_\ical &\epseq \sum_{j \,\in\, \ical} I_j + \sum_{\substack{|e
     \,\cap\, \ical| \,\ge\, 1, \\ e \,\subset\, \Ss}}  I_e + 
 \sum_{j \,\in\, \ical^c} I_j + \sum_{\substack{|e \,\cap\, \ical| =
     1, \\|e \,\cap\, \Ss^c| =1 }} I_e + \sum_{| e\, \cap \, \ical| =
   0} I_e \\
  &= \sum_{j \in V} I_j + \sum_{ e \in E} I_e.
\end{align*}
Since this expression does not depend on $\ical$, using log-sum-max rule~\Rule{\ref{itm:log:sum:max}} as before, we obtain that $\frac1n \log \mixR_\phi^{k,n} \epseq
\sum_{j \in V} I_j + \sum_{ e \in E} I_e $.

Now, we need to analyze $\mixR_\phi^{\infty,n}$. We try to break up the sum
as before into $\Ttermt_\ical$ terms (defined similar to
$\Tterm_\ical$ for $\mixR_\phi^{k,n}$). This time however, we only
need to consider $\ical = \Ss$ (and $\lcal$ the empty set), because $\phi = \infty$ implies
$\lam_j = \infty$ for all $j \in \Ss$. The expansion for
$\Ttermt_\Ss$ can be obtained from~(\ref{eq:Tj:expan:long}) by
setting $k = \infty$ and removing the terms corresponding to indices in
$\Ss = \ical \cup \lcal$, 
\begin{align*}
  \mixR_\phi^{\infty,n} = \Ttermt_\Ss = \sum_{k_j \,\ge\, 1, \; j \in \Ss^c}
  \Big\{ A e^{ -\sum_{j \in \Ss^c} \beta_j k_j} \prod_{j \in \Ss^c}
  R_{k_j}^n \{j\} \prod_{\substack{|e \,\cap\, \Ss| = 1, \\e \,\cap\, \Ss^c =
      \{\ell\}}} R_{k_\ell}^n(e) \prod_{ |e \,\cap\, \Ss| = 0}
  R_{k_e}^n(e) \Big\}.
\end{align*}
It follows that
\begin{align*}
  \frac{1}{n} \log \mixR_\phi^{\infty,n} \epseq\;
  \sum_{j \in \Ss^c} I_j + 
   \sum_{\substack{|e \,\cap\, \Ss| = 1, \\|e \,\cap\, \Ss^c| =
      1}}  I_e + \sum_{ |e \,\cap\, \Ss| = 0} I_e.
\end{align*}
The last two sums can be described as the sum over all edges
$e : e \cap \Ss^c \neq \emptyset$. Putting the pieces together, we
have
\begin{align*}
  \frac1n\log \Dphi^{k,n} &\epseq \Big( \sum_{j \in V} I_j + \sum_{ e
    \in E} I_e \Big) - \Big( \sum_{j \in \Ss^c} I_j + \sum_{e \,\cap\,
    \Ss^c \neq \emptyset} I_e \Big)  \\
  &= \sum_{j \in \Ss} I_j + \sum_{e \,\subset\, \Ss} I_e
\end{align*}
as desired.

\section{Proof of Theorem~\ref{thm:epseq:R:S:M:L}}
\label{sec:epseq:proofs}
Let us start by understanding the asymptotic behavior of $\frac1n \log
R_\mui^n(e)$. Throughout, we fix $e \in \Et$. We either have $e =
\{j\}$ in which case $\lam_e = \lam_j$, or $e = \{i,j\}$ in which case
$\lam_e = \lam_i \wedge \lam_j$. Recall that $\ms = (m_1,\dots,m_d)
\in \nats^d$ is a multi-index, and we will work under the collection
$\{\pr_{\lams}^{\ms}\}$ of conditional distributions (see
Definition~\ref{def:epseq} for details). The same convention is used regarding the meaning of $m_e$, that is, $m_e = m_j$ for $e = j$, and $m_e = m_i \wedge m_j$ for $e = \{i,j\}$. We also fix some $k\in \nats$, which is the parameter $k$ appearing in Definition~\ref{def:epseq} (reserved for the ultimate conditioning on $\{\phi = k\}$). Finally, we always assume $\eps \in (0,1)$.

At first, we need to be careful about whether $\mui < m_e$ or $\mui \ge m_e$.
\begin{lem}\label{lem:Hoeff:R}
	Let $\mui \in [n]$ and assume $\mui \ge m_e$. Then,
	\begin{align}\label{eq:Hoeff:R}
		\pr_{\lams}^{\ms} \Big(
		\Big| \frac{1}{n-\mui+1} \log R_\mui^n(e) - I_e 
			\Big| > \eps \Big) \le 
		2 \exp \Big[ {-
		\frac{(n-\mui+1)\eps^2}{2M^2}}\Big]
	\end{align}
\end{lem}
\begin{proof}
	Since $m_e \le \mui$, conditioned on $\pr_{\lams}^{\ms}$, $\Xb_e^\mui, \Xb_e^{\mui+1},\dots$ are i.i.d. from $f_e$. Recalling definition~\eqref{eq:R:def}, $\log R_\mui^n(e) = \sum_{t = \mui}^n h_e(\Xb_e^t)$ which is a sum of $(n-\mui+1)$ i.i.d. bounded variables $h_e(\Xb_e^t) \in [-M,M]$ with mean $\ex_{f_e} h_e(\Xb_e^\mui) = I_e$. The result then follows from Hoeffding inequality. 	
\end{proof}

Before moving on, we need an extension of
Definition~\ref{def:epseq}. We need to deal with intermediate
sequences whose terms depend possibly on $\ms$ (in addition to
$k$). There is nothing to preclude such dependence in
Definition~\ref{def:epseq}. Hence, we use the same definition for
$\eps$-equivalence of such sequences with respect to the collection
$\{\pr_{\lams}^{\ms}\}$.  Note that for any $\mui, \nui \in \nats$, we can write
\begin{align}\label{eq:R:mult:break}
R_\mui^n(e) = R_\mui^{\nui-1}(e) R_{\mui \vee \nui}^n(e)
\end{align}
which holds irrespective of whether $\mui \ge \nui$ or $\mui < \nui$.


\begin{lem}\label{lem:R:info:part}
	For any $\mui \in [k]$, $\frac{1}{n} \log R_{\mui \vee m_e}^n(e) \epseq I_e$ as $n \to \infty$ with respect to $\{\pr_{\lams}^{\ms}\}$
\end{lem}

\begin{lem}\label{lem:R:zero:part}
	For any $\mui \in \nats$,  $\frac{1}{n} \log R_\mui^{m_e-1}(e) \epseq 0$ as $n \to \infty$ with respect to $\{\pr_{\lams}^{\ms}\}$.
\end{lem}

\begin{lem}\label{lem:R:final}
	For any $\mui \in [k]$, $\frac{1}{n} \log R_\mui^n(e) \epseq I_e$ as $n \to \infty$ with respect to $\{\pr_{\lams}^{\ms}\}$.
\end{lem}

The last lemma proves the statement in Theorem~\ref{thm:epseq:R:S:M:L}
regarding asymptotic behavior of $\frac{1}{n} \log R_\mui^n(e)$ for
$\mui \in [k]$.

\begin{proof}[Proof of Lemma~\ref{lem:R:info:part}]
Apply Lemma~\ref{lem:Hoeff:R} with $\mui$ replaced with $\mui \vee m_e$. Since $\eps < 1$ and $\mui \vee m_e \le k \vee m_e$, the RHS of~\eqref{eq:Hoeff:R} is further bounded above by
\begin{align*}
	2 \exp \Big( \frac{(k \vee m_e) \eps}{2M^2} \Big)
		\exp \Big( {-\frac{n\eps^2}{2M^2}} \Big) \le 
			2 \exp \Big( {-\frac{n\eps^2}{4M^2}} \Big)
\end{align*}
as long as $(k \vee m_e) \eps \le n\eps^2/2$ or equivalently $ n\eps \ge 2 (k \vee m_e)$. (This same condition guarantees $\mui \vee m_e \in [n]$ justifying application of Lemma~\ref{lem:Hoeff:R}.) The condition obtained is of  the form required by Definition~\ref{def:epseq}, since $2( k \vee m_e )$ is bounded above by a polynomial, say $2(k + m_i)$ if $e = \{i,j\}$. This shows that 
\begin{align*}
\frac{1}{n-\mui \vee m_e+1} \log R_{\mui \vee m_e}^n(e) \epseq I_e \quad \text{ w.r.t.} \;\{ \pr_{\lams}^{\ms}\}.
\end{align*}

Now, note that $| \frac{n-\mui \vee m_e+1}{n} - 1| \le \frac{\mui \vee m_e-1}{n} \le \frac{k \vee m_e}{n}$ which can be made $\le \eps$ by choosing $n \eps \ge (k \vee m_e)$. This implies that $\frac{n-\mui \vee m_e+1}{n} \epseq 1$. Applying rule~\Rule{\ref{itm:scale:one}}, with $a_n = \frac{1}{n-\mui \vee m_e+1} \log R_{\mui \vee m_e}^n(e)$, $b_n = I_e$ and $c_n = \frac{n-\mui \vee m_e+1}{n}$, we obtain the desired result.
\end{proof}

\begin{proof}[Proof of Lemma~\ref{lem:R:zero:part}]
If $\mui > m_e - 1$, we have by definition $R_\mui^{m_e-1}(e) = 1$ and there is nothing to show. Otherwise, 
by boundedness assumption $\|h\|_\infty \le M$, we have
\begin{align*}
e^{-M m_e}\le e^{-M(m_e-\mui)}\le R_{\mui}^{m_e-1}(e) \le e^{M(m_e-\mui)} \le e^{M m_e }.
\end{align*}
Hence, by taking $n \eps \ge M m_e$, we have $|\frac{1}{n} \log R_\mui^{m_e -1}(e)| \le \eps$, which implies the result.
\end{proof}

\begin{proof}[Proof of Lemma~\ref{lem:R:final}]
	Apply~\eqref{eq:R:mult:break} with $\nui = m_e$, to obtain
\begin{align*}
	\frac1n \log R_\mui^n(e) = 
	\frac1n \log R_\mui^{m_e-1}(e) + 
	\frac1n \log R_{\mui \vee m_e}^n(e).
\end{align*}	
The result now follows from Lemmas~\ref{lem:R:info:part} and~\ref{lem:R:zero:part} and rule~\Rule{\ref{itm:sum}}.
\end{proof}

\subsection{Bounding $\Sterm$-terms}
Bounding $\Sterm$-terms is perhaps the most elaborate part of the
proof. We start with a uniformization of Lemma~\ref{lem:Hoeff:R} and
then proceed in steps, working on various parts of the sum $\Sterm_\mui^{\infty\cn}(e)
:= S_\mui^{\infty,n}(e)$ one at a time. Up to
Lemma~\ref{lem:S:degen:part}, we will use the
shorthand notation introduced in~\eqref{eq:short:notation}, with $n$ superscript
dropped.
It might help to recall that in this notation, $\mui$ and $\infty$ are
the initial and final indices of the sum, respectively. Also, the edge
$e \in E$ is fixed throughout.

\begin{lem}\label{lem:S:sup:devia}
	Let $\mui \in \nats$ and $\alpha \in (0,1)$ such that $m_e \le \mui \le \floor{\alpha n}$. Then,
	\begin{align*}
	\sup_{\mui \,\le\, p \,\le\, \floor{\alpha n}}
		\Big| 
		\frac{1}{n-p+1} \log R_p^n(e) - I_e
		\Big| \le \eps
	\end{align*}
	with $\pr_{\lams}^{\ms}$-probability at least $1-2( \floor{\alpha n}- \mui +1) \exp [ -\frac{\eps^2((1-\alpha)n+1)}{2M^2}]$.	
\end{lem}

\begin{lem}\label{lem:S:upto:alpha}
	Let $\mui \in [k]$ and $\alpha \in (0,1)$ such that $m_e \le \mui \le \floor{\alpha n}$. Then
	\begin{align*}
		\pr_{\lams}^{\ms} \Big(
			\Big|
		\frac{1}{n} 
		\log S_\mui^{\floor{\alpha n}\cn} (e) - I_e\Big| \le 2 \eps
		\Big) \ge 1 - 2 n 
		\exp\Big({-\frac{1-\alpha}{2M^2} n \eps^2}\Big)
	\end{align*}
	for $n \eps \ge c_0 k$ and $\eps \in (0,1)$.
\end{lem}

\begin{lem}\label{lem:S:beyond:alpha}
	Let $\delta \in (0,1)$ and $\alpha = \frac{M + \delta \beta}{M + \beta}$. Then, for $n \ge n_0(A,\beta,M,\delta)$,
	\begin{align*}
		\frac{1}{n} \log 
			\Sterm_{\floor{\alpha n} + 1}^{n\cn}(e) 
		\le -\frac{\delta}{2} \beta.
	\end{align*}
\end{lem}

\begin{lem}\label{lem:S:frac:n:upper}
	Let $\alpha = \frac{M + \frac12 \beta}{M + \beta}$ and $\mui \in [k]$. Then,  $\frac1n \log \Sterm_{\mui \vee m_e}^{\floor{\alpha n}\cn}(e) \epseq I_e$ as $n \to \infty$ w.r.t. $\{\pr_{\lams}^{\ms}\}$.
\end{lem}

\begin{lem}\label{lem:S:xi:n:n}
	For any $\mui \in [k]$, we have $\frac1n \log \Sterm_{\mui \vee m_e}^{n\cn} (e) \epseq I_e$ as $n \to \infty$ w.r.t. $\{\pr_{\lams}^{\ms}\}$.
\end{lem}

\begin{proof}[Proof of Lemma~\ref{lem:S:sup:devia}]
	We note that for any $p = \mui, \mui+1, \dots, \floor{\alpha n}$, Lemma~\ref{lem:Hoeff:R} applies.  We can further upper-bound the RHS of~\eqref{eq:Hoeff:R} by
	\begin{align*}
		2 \exp \Big[
			{-\frac{(n-p+1)\eps^2}{2M^2}}
			\Big] 
		\le 2 \exp \Big[
			{-\frac{(n-\alpha n+1)\eps^2}{2M^2}}
			\Big].
	\end{align*}
	The result follows by applying union bound.
\end{proof}

\begin{proof}[Proof of Lemma~\ref{lem:S:upto:alpha}]
	By Lemma~\ref{lem:S:sup:devia}, uniformly over $p = \mui,\dots,\floor{\alpha n}$, we have
	\begin{align}\label{eq:temp1}
		e^{(n-p+1)(I_e - \eps)} \le R_p^n(e)
		\le e^{(n-p+1)(I_e + \eps)}
	\end{align}
	with $\pr_{\lams}^{\ms}$-probability at least
	$1 - 2 n \exp(- \frac{(1-\alpha)n \eps^2}{2M^2})$. (Note that this is a further lower bound w.r.t. that of Lemma~\ref{lem:S:sup:devia}) On the event that~\eqref{eq:temp1} holds, we have
	\begin{align*}
		S_\mui^{\floor{\alpha n}\cn}(e)
		&\le \sum_{p = \mui}^{\floor{\alpha n}}
			A e^{-\beta p} e^{(n-p+1)(I_e + \eps)} 
			\\
		&\le A e^{(n+1)(I_e + \eps)} 
			\sum_{p = 1}^{\infty} 
			e^{-(\beta + I_e)p} 
			=  \frac{A e^{(n+1)(I_e + \eps)}}
					{e^{\beta + I_e} -1}.
	\end{align*}
	Take $C_1 := \max\{0, \log \frac{A}{e^{\beta + I_e} -1}\}$. Then,
	\begin{align*}
		\frac{1}{n} \log S_\mui^{\floor{\alpha n}\cn}(e)
			\le \frac{C_1}{n} + \frac{n+1}{n} (I_e + \eps) 
			\le I_e + 2 \eps
	\end{align*}
	as long as $n \eps \ge C_1 + I_e + 1$ (and $\eps < 1$). To get the lower bound, we note that~\eqref{eq:temp1} implies
	\begin{align*}
		S_\mui^{\floor{\alpha n}\cn}(e)
			&\ge \sum_{p = \mui}^{\floor{\alpha n}}
			A e^{-\beta p} e^{(n-p+1)(I_e - \eps)} \\
			&\ge A e^{(n+1)(I_e - \eps)} e^{-(\beta+I_e)\mui }
	\end{align*}
	where we have lower bounded a sum of nonnegative terms by its first term. Hence,
	\begin{align*}
		\frac{1}{n} \log S_\mui^{\floor{\alpha n}\cn}(e)
		\ge \frac{n+1}{n}(I_e -\eps) - 
			\frac{|\log A| + (\beta + I_e)\mui}{n} \\
		\ge I_e - \eps  -
			\frac{1+ |\log A| + (\beta + I_e)k}{n} 
			\ge I_e - 2\eps 
	\end{align*} 
	as long as $n \eps \ge 1 + |\log A| + (\beta + I_e)k$.
\end{proof}

\begin{proof}[Proof of Lemma~\ref{lem:S:beyond:alpha}]
	By boundedness assumption $\|h\|_\infty \le M$, we have $R_p^n(e) \le e^{(n-p+1)M}$ as long as $p \le n$. Hence,
	\begin{align*}
		\Sterm_{\floor{\alpha n} + 1}^{n\cn}(e) 
		&\le \sum_{p = \floor{\alpha n}+1}^{n}
			A e^{-\beta p} e^{(n-p+1)M} \\
		&= A e^{(n+1)M} \sum_{p = \floor{\alpha n}+1}^{n}
		e^{-(\beta + M)p} \\
		&\le A e^{(n+1)M} (n - \alpha n + 1)
		e^{-(\beta + M)\alpha n}
	\end{align*}
	where we have used $\floor{\alpha n} > \alpha n -1$. Taking $\alpha$ to be as stated and noting that $1-\alpha \in (0,1)$, we get
	\begin{align*}
		\frac1n \Sterm_{\floor{\alpha n} + 1}^{n\cn}(e) 
		\le \frac{|\log A|}{n} + \frac{n+1}{n} M + 
		\frac{\log((1-\alpha)n + 1)}{n} -M - \delta \beta
		\le - \frac{\delta}{2} \beta
	\end{align*}
	as long as $n \ge n_0(A,\beta,M,\delta)$ for some $n_0$ large enough.
\end{proof}

\begin{proof}[Proof of Lemma~\ref{lem:S:frac:n:upper}]
	Apply Lemma~\ref{lem:S:upto:alpha} with $\mui$ replaced with $\mui \vee m_e$. To ensure $\mui \vee m_e \le \floor{\alpha n}$, let $n \ge \frac{1}{\alpha} (k \vee m_e+1)$. To ensure that  the bound of Lemma~\ref{lem:S:upto:alpha} holds, let $n \eps \ge c_0 k$. Since, these two conditions are met if $n \eps \ge  \frac{1}{\alpha} (k + m_e) + c_0 k$, the result follows. (Note also that $\frac{1-\alpha}{2M^2}$ is a positive constant by our choice of $\alpha$.)
\end{proof}

\begin{proof}[Proof of Lemma~\ref{lem:S:xi:n:n}]
	Let $\alpha = \frac{M + \frac12 \beta}{M + \beta}$ and as in Lemma~\ref{lem:S:frac:n:upper} assume $n \ge \frac{1}{\alpha} (k \vee m_e+1)$ so that $\mui \vee m_e \le \floor{\alpha n}$. (This is just to make sure that sums ranging from $\mui \vee m_e$ to $\floor{\alpha n}$ are not vacuous.) By Lemma~\ref{lem:S:upto:alpha}, we have
	\begin{align*}
		\frac1n \log \Sterm_{\floor{\alpha n} + 1}^{n\cn}(e) 
			\epsle -\frac14 \beta
	\end{align*}
	and by Lemma~\ref{lem:S:frac:n:upper}, $\frac1n \log \Sterm_{\mui \vee m_e}^{\floor{\alpha n}\cn}(e) \epseq I_e$. Now, we can break up the sum and use log-sum-max rule~\Rule{\ref{itm:log:sum:max}},
	\begin{align*}
		\frac1n \log \Sterm_{\mui \vee m_e}^{n\cn} (e) 
		&= \frac1n \log \Big(
			\Sterm_{\mui \vee m_e}^{\floor{\alpha n}\cn} (e)
			+ \Sterm_{\floor{\alpha n} + 1}^{n\cn}(e) 
			\Big)\\
		&\epseq 
			  \max \Big\{
		\underbrace{ \frac1n \log 
		\Sterm_{\mui \vee m_e}^{\floor{\alpha n}\cn}(e)}%
		_{\epseq I_e},\;
		\underbrace{\frac1n \log
			\Sterm_{\floor{\alpha n} + 1}^{n\cn}(e)}%
		_{\epsle -\frac14 \beta}	 
			\Big\} 
		\epseq I_e
	\end{align*}
	where the last $\epseq$ follows from rule~\Rule{\ref{itm:eq:ineq}}.
\end{proof}

The next step is to move from $S_{\mui \vee m_e}^{n\cn}(e)$ to $S_{\mui}^{n\cn}(e)$. We need a couple of lemmas. To simplify notation, throughout this section, let
\begin{align}\label{eq:xi:def}
	\xi := \xi(\mui,m_e) := \mui \vee m_e.
\end{align}
We occasionally drop the dependence of $\xi$ on $\mui$ and $m_e$
(although this is implicitly assumed). We note that all the lemmas
established so far in this section hold, if we replace $[k]$ in their
statements with $[2k]$ (or any other constant multiple of $k$). For
the rest of this subsection, we will use the full superscript notation
$\Sterm_\mui^{\nu,n}(e)$ introduced in~(\ref{eq:Sterm:def}).
%
%
%

\begin{lem}\label{lem:S:degen:part}
	For $1\neq \mui \in [2k]$, we have $\frac1n \log \Sterm_{\mui-1}^{\xi-1,	
	\xi -1} (e) \epseq 0$ as $n \to \infty$ w.r.t. $\{\pr_{\lams}^{\ms}\}$.
\end{lem}

\begin{lem}\label{lem:S:mu:xi:n}
	For $1 \neq \mui \in [2k]$, we have
	$\frac1n \log \Sterm_{\mui-1}^{\xi-1,n} (e) \epseq I_e$ as $n \to \infty$ w.r.t. $\{\pr_{\lams}^{\ms}\}$.
\end{lem}

\begin{lem}\label{lem:S:mu:n:n}
	For $1 \neq \mui \in [2k]$, we have $\frac1n \log \Sterm_{\mui-1}^{n,n} (e) \epseq I_e$ as $n \to \infty$ w.r.t. $\{\pr_{\lams}^{\ms}\}$.
\end{lem}

\begin{lem}\label{lem:S:mu:inf:n}
	For $\mui \in [k]$, we have $\frac1n \log \Sterm_\mui^{\infty,n}(e) \epseq I_e$ as $n \to \infty$ w.r.t. $\{\pr_{\lams}^{\ms}\}$.
\end{lem}

The last lemma completes the proof of the statement in Theorem~\ref{thm:epseq:R:S:M:L} regarding the $\Sterm$ terms.

\begin{proof}[Proof of Lemma~\ref{lem:S:degen:part}]
	For $p \le \xi-1$,  $ e^{-M(\xi-p)} \le R_p^{\xi-1}(e) \le e^{M(\xi-p)}$, by boundedness assumption. Hence, we have
	\begin{align*}	
		S_{\mui-1}^{\xi-1,\xi-1}(e) &\le 
			\sum_{p = \mui-1}^{\xi-1} 
			A e^{-\beta p} e^{M(\xi-p)} \le 
			\frac{A e^{M\xi} }{1 - e^{-(\beta+M)}},
	\\
		S_{\mui-1}^{\xi-1,\xi-1}(e) &\ge
		 \sum_{p=\mui-1}^{\xi-1} 
		 A e^{-\beta p} e^{-M(\xi-p)}	
		 \ge A e^{-\beta(\xi-1)} e^{-M}
	\end{align*}
	Let $C_1 = |\log \frac{A}{1-e^{-(\beta+M)}}|$ and $C_2 = |\log(Ae^{\beta - M})|$. We have
	\begin{align*}
		-\frac{C_2}{n} 
		- \frac{2\beta(k \vee m_e)}{n} \,\le\, 
		\frac1n \log S_{\mui-1}^{\xi-1,\xi-1}(e)
		\,\le\, \frac{C_1}{n} + \frac{2M( k \vee m_e)}{n}
	\end{align*}
	where we have used $\xi \le (2k) \vee m_e$ which follows from definition~\ref{eq:xi:def} and assumption $\mui \in [2k]$. It follows that $|\frac1n \log S_{\mui-1}^{\xi-1,\xi-1}(e)| \le \eps$ if we take $n \eps \ge C_3 (k+m_e)$, proving the result.
\end{proof}

\begin{proof}[Proof of Lemma~\ref{lem:S:mu:xi:n}]
For any $p \in \{\mui-1,\dots,\xi -1\}$, we have by~\eqref{eq:R:mult:break},
\begin{align*}
	R_p^n(e) = R_p^{\xi -1}(e)
	R_{\xi}^n(e).
\end{align*}
It follows from the definition of $\Sterm$ term that
\begin{align*}
	\Sterm_{\mui-1}^{\xi-1,n}(e)
	 = R_{\xi}^n(e) 
	 \sum_{p = \mui-1}^{\xi - 1} 
	 A e^{-\beta p} R_p^{\xi -1}(e)	
	 = R_\xi^n(e) \,\Sterm_{\mui-1}^{\xi-1,\xi-1}(e).
\end{align*}
The result now follows from Lemmas~\ref{lem:R:info:part} and~\ref{lem:S:degen:part}.
\end{proof}


\begin{proof}[Proof of Lemma~\ref{lem:S:mu:n:n}]
	We have $\Sterm_{\mui-1}^{n,n}(e) = \Sterm_{\mui-1}^{\xi-1,n}(e) + \Sterm_{\xi}^{n,n}(e)$. The result now follows form Lemmas~\ref{lem:S:mu:xi:n} and~\ref{lem:S:xi:n:n}, and log-sum-max rule~\Rule{\ref{itm:log:sum:max}}.
\end{proof}

Note that since $[k+1] \subset [2k]$, it follows that $\frac1n \log \Sterm_{\mui}^{n,n} (e) \epseq I_e$ for all $\mui \in [k]$. The final step is to move from $\Sterm_\mui^{n,n}(e)$ to $\Sterm_\mui^{\infty,n}(e)$.


\begin{proof}[Proof of Lemma~\ref{lem:S:mu:inf:n}]
	We have $\Sterm_\mui^{\infty,n}(e) = \Sterm_\mui^{n,n}(e) + \Sterm_{n+1}^{\infty,n}(e)$. Since $R_p^n(e) = 1$ for all $p > n$ (by convention), we have
	\begin{align*}
		\Sterm_{n+1}^{\infty,n}(e) =
		\sum_{p=n+1}^\infty A e^{-\beta p} = 
		\frac{A e^{-\beta(n+1)}}{1-e^{-\beta}} .
	\end{align*}
	It follows that $\frac1n \log \Sterm_{n+1}^{\infty,n}(e) \epseq -\beta$. Then, by rules~\Rule{\ref{itm:log:sum:max}} and~\Rule{\ref{itm:max}},
	\begin{align*}
	\frac1n \log  \Sterm_\mui^{\infty,n}(e) 
		&\epseq
		\max\Big\{
			\frac1n \log \Sterm_\mui^{n,n}(e) \;,
			\frac1n \log \Sterm_{n+1}^{\infty,n}(e) 
		\Big\} \\
		&\epseq \max\{ I_e, -\beta\} = I_e
	\end{align*}
	where we have used Lemma~\ref{lem:S:mu:n:n}.
\end{proof}

\subsection{Bounding $\Mterm$-terms}
With some work, we can reduce bounding $\Mterm$-terms to that of bounding $R$ and $\Sterm$-terms.

\begin{lem}\label{lem:M:mu:n:n}
	For $\mui \in [k]$, we have $\Mterm_\mui^{n\cn}(e) \epseq I_e$ as $n \to \infty$ w.r.t $\{\pr_{\lams}^{\ms}\}$.
\end{lem}
\begin{proof}
	Let $n \ge k$, so that the sums are not vacuous. For $q \in \{\mui,\dots,n\}$ the cardinality of the set $\{(p_1,p_2) : \mui \le p_1,p_2 \le n, \; p_1 \wedge p_2 = q\}$ is $2(n-q)+1$. Hence,
	\begin{align*}
		\Mterm_\mui^{n\cn}(e) &\le
			\sum_{p_1 = \mui}^n 
			\sum_{p_2 = \mui}^n
			A e^{-(\beta_1 \wedge \beta_2) 
				(p_1 \wedge p_2)} 
				R_{p_1 \wedge p_2}^n(e) \\
			&= A\sum_{q=\mui}^n
				[2(n-q) + 1]e^{-(\beta_1 \wedge \beta_2)q}
				R_q^n(e) \\
			&\le  n \sum_{q=\mui}^n
			 2A e^{-(\beta_1 \wedge \beta_2)q} R_q^n(e).
	\end{align*}
	Note that this last sum is of the form $\Sterm_\mui^{n\cn}(e)$. For the lower bound, we use the first term of the sum, $\Mterm_\mui^{n\cn}(e) \ge A e^{-(\beta_1+\beta_2)\mui} R_\mui^n(e)$. Since $\mui \le k$, we have
	\begin{align*}
		-\frac{|\log A|+(\beta_1 + \beta_2)k}{n}
		+ \frac1n \log R_\mui^n(e) \,\le\, 		
		\frac{1}{n} \log \Mterm_\mui^{n\cn}(e) 
			\,\le\,
		\frac{\log n}{n} + \frac1n \log 
		\Sterm_\mui^{n\cn}(e).
	\end{align*}
	The only new term (with respect to what established earlier) is $\log n /n$ which is $\epseq 0$. This can be seen by noting that $|\log n /n | \le \eps$ if $\sqrt{n}\eps \ge 1$. The result now follows from Lemmas~\ref{lem:R:final} and~\ref{lem:S:mu:inf:n}.
\end{proof}

To move from $\Mterm_\mui^{n\cn}(e)$ to $\Mterm_\mui^{\infty\cn}(e)$, we introduce the following extended notation
\begin{align}\label{eq:extnd:Mterm:def}
	\Mterm_{a,b}^{c,d\cn}(e) := \Mterm_{a,b}^{c,d,n}(e)
		:= \sum_{p_1 = a}^c \sum_{p_2=b}^{d}
		A e^{-(\beta_1 p_1+ \beta_2 p_2)}
		R_{p_1 \wedge p_2}^n(e)
\end{align}
so that $\Mterm_{\mui}^{\infty\cn}(e) = \Mterm_{\mui,\mui}^{\infty,\infty\cn}(e)$.

\begin{lem}\label{lem:M:mu:inf:n}
	For $\mui \in [k]$, we have $\Mterm_\mui^{\infty\cn}(e) \epseq I_e$ as $n \to \infty$ w.r.t. $\{\pr_{\lams}^{\ms}\}$.
\end{lem}
\begin{proof}
	Let $n \ge k$. The strategy is to break up the sum as 
	\begin{align}\label{eq:4:sum:break}
		\Mterm_{\mui,\mui}^{\infty,\infty\cn}(e)
		=
		\Mterm_{\mui,\mui}^{n,n\cn}(e) +
		\Mterm_{n+1,\mui}^{\infty,n\cn}(e) +
		\Mterm_{\mui,n+1}^{n,\infty\cn}(e) +
		\Mterm_{n+1,n+1}^{\infty,\infty\cn}(e)
	\end{align}
	and then apply the log-sum-max rule~\Rule{\ref{itm:log:sum:max}}. The first term is taken care of by Lemma~\ref{lem:M:mu:n:n}. For the second term, we have
	\begin{align*}
		\Mterm_{n+1,\mui}^{\infty,n\cn}(e) &=
		\sum_{p_1 = n+1}^\infty \sum_{p_2=\mui}^{n}
		A e^{-(\beta_1 p_1+ \beta_2 p_2)}
		R_{p_2}^n(e) \\
		&= C_1 e^{-\beta_1(n+1)} \Sterm_\mui^{n\cn}(e)
	\end{align*}
	Applying Lemma~\ref{lem:S:mu:n:n} we get $\frac1n \log \Mterm_{n+1,\mui}^{\infty,n\cn}(e) \epseq -\beta_1 + I_e$. The third term in~\ref{eq:4:sum:break} is similar. Recalling that $R_p^n(e) = 1$ for $p > n$, the fourth term, $\Mterm_{n+1,n+1}^{\infty,\infty\cn}(e)$, is equal to $C_2 e^{-(\beta_1 + \beta_2)(n+1)}$. Hence, by \Rule{\ref{itm:log:sum:max}},
	\begin{align*}
		\frac{1}{n} \log \Mterm_{\mui,\mui}^{\infty,\infty\cn}(e) \,\epseq\, 
		\max\{ I_e, -\beta_1 + I_e, -\beta_2 + I_e, -\beta_1 -\beta_2\} = I_e.
	\end{align*}
\end{proof}

\subsection{Bounding $\Lterm$-terms}
\begin{lem}
	For $\mui,r \in [k]$, we have $\frac1n \log \Lterm_{\mui,(r)}^{\infty\cn} \epseq I_e$ as $n \to \infty$ w.r.t. $\{\pr_{\lams}^{\ms}\}$.
\end{lem}
\begin{proof}
	First consider the case $\mui \ge r$. Then, we have
	\begin{align*}
		\Lterm_{\mui,(r)}^{\infty\cn} = 	
		\sum_{p=\mui}^\infty A e^{-\beta p} R_{r}^n(e)
		=  C_1 e^{-\beta \mui} R_{r}^n(e) .
	\end{align*}		
	Since $| \frac{\beta \mui}{n}| \le \frac{\beta k}{n}$, we have $\frac1n \log (C_1 e^{-\beta \mui}) \epseq 0$. The result now follows from Lemma~\ref{lem:R:final}. For the case $\mui < r$, we have
	\begin{align}
		\Lterm_{\mui,(r)}^{\infty\cn} &=
		\sum_{p=\mui}^{r-1} A e^{-\beta p} R_p^n(e) +
		\sum_{p=r}^\infty A e^{-\beta p} R_r^n(e)
		\notag \\
		&=\Sterm_{\mui}^{r-1\cn}(e) + C_1 e^{-\beta r}
		R_r^n(e).\label{eq:L:temp}
	\end{align}
	Let $n \ge k$ so that $n \ge r$. Note that $Ae^{-\mui \beta} R_\mui^n(e) \le \Sterm_\mui^{r-1\cn}(e) \le \Sterm_\mui^{n\cn}(e)$. It follows from Lemmas~\ref{lem:S:mu:n:n} and~\ref{lem:R:final}, and $\frac{1}{n} \log (A e^{-\beta \mui}) \epseq 0$ that $\frac{1}{n} \log \Sterm_{\mui}^{r-1\cn}(e) \epseq I_e$. Applying rule~\Rule{\ref{itm:log:sum:max}} to~\eqref{eq:L:temp} and using a similar argument for the second term, we get the result.
\end{proof}


\section{Conclusion}
We have introduced a graphical model framework which allows for modeling and detection of multiple change points in networks. Within this framework, we proposed stopping rules for the detection of change points and particular functionals of them (the minimum over a subset), based on thresholding the posterior probabilities. A message passing algorithm for efficient computation of these posteriors was derived. It was also shown that the proposed rules are asymptotically optimal in terms of their expected delay, within the Bayesian framework.

Let us discuss some directions for possible extension of this
work. The assumption that the distribution of shared (edge)
information between two nodes only depends on the minimum of the
associated change points (cf. discussion after
equation~\eqref{eq:joint:def}) might be restrictive in practice. The
current assumption simplifies the analysis in many places and it has
an impact on the asymptotic delay. For example, we suspect that the ``no gain'' phenomenon in asymptotic delay for detection of a single change point, discussed in Remark~3 after Theorem~\ref{thm:asym:opt}, is due to this rather simplistic assumption. It will be interesting to be able to extend the analysis to a model which allows for a more general dependence on the two change points. At present, however, we do not know how much of our analysis can be carried over to the general case.

It is possible to derive an approximate message passing algorithm with computational cost scaling as $O(|V| + |E|)$ for each time step $n$. That is, the computational cost is constant in time $n$. Simulations indicate that this fast algorithm approximates the exact message passing well. The presentation of the algorithm and its theoretical analysis will be deferred to a future publication.

As was discussed in the remarks after Theorems~\ref{thm:asym:opt}
and~\ref{thm:epseq:R:S:M:L}, the assumptions on the  likelihood ratio,
i.e., the boundedness, and the priors, i.e., exponential tail decay
are crucial to our proof. They seem to strike the right balance between the prior and the likelihood and they also allow for the break-up of the analysis of the rather complicated likelihood ratios (cf. \eqref{eq:mixR:def}) into simpler pieces. This is in contrast to the more classical case of a single change point where the analysis goes through seamlessly, say, irrespective of the tail behavior of the priors~\cite{TarVee05}. Whether these limitations are genuinely present in the multiple change point model or are artifacts of the proof technique is not clear at this point. 

Finally, although our main focus in this paper was on the Bayesian formulation, we note that there are non-Bayesian optimality criteria for the single-change point problem, e.g., the minimax as considered in~\cite{TarPol11}. It is an interesting question whether one can derive minimax optimal rules for the model we consider here.

%

\appendix

\section{Proof of Lemma~\ref{lem:const}}

\label{app:lem:const}

  Consider, for example, node $i_1$ and let $j$ be one of its
  neighbors in $G$, i.e. $\{i_1,j\} \in E$. Let $k_1 \ge n + 1$. Then
    $P(\Xb_{i_1}^n | \lambda_{i_1} = k_1) = \prod_{t = 1}^n
    g_{i_1}(X_{i_1}^t) = P(\Xb_{i_1}^n | \lambda_{i_1} =
    n+1)$. Similarly, the distribution of $X_{i_1 j}$ given
    $\lambda_{i_1} = k_1$ and $\lambda_j$ is independent of the
    particular value of $k_1$, that is,
    \begin{align*}
      P(\Xb_{i_1 j}^n | \lambda_{i_1} = k_1, \lambda_j) =
      P(\Xb_{i_1 j}^n | (n+1) \wedge \lambda_j) = 
      P(\Xb_{i_1 j}^n | \lambda_{i_1} = n+1, \lambda_j).
    \end{align*}
 
  Let $\ical := \{i_1,\dots,i_r\}$ and $\ical^c = [d] \setminus \ical
  $. Pick $k_j \ge n+1$ for $j \in \ical$. Then, the argument above applied to each node in $\ical$ shows that
  \begin{align*}
    P(\Xbs^n \,|\, \lambda_j = k_j, j \in \ical) &= 
    \sum_{\lams} P(\Xbs^n \,|\, \lams)\, 
    P(\lams \,|\, \lambda_j = k_j, j \in  \ical) \\
    &= \sum_{\lambda_\ell, \,\ell \,\in\, \ical^c}
    P(\Xbs^n \, | \, \lambda_{\ell}, \ell \in \ical^c,\, \lambda_{j}
    = k_j, j \in \ical) \; 
    P(\lambda_\ell, \ell \in \ical^c) \\
    &= \sum_{\lambda_\ell, \,\ell \,\in\, \ical^c}
    P(\Xbs^n \, | \, \lambda_{\ell}, \ell \in \ical^c,\, \lambda_{j}
    = n+1, j \in \ical) \; 
    P(\lambda_\ell, \ell \in \ical^c)
  \end{align*}
  where the second inequality follows by independence of
  $\{\lambda_i\}$ a priori. As the last expression does not depend on
  $\{k_j\}$, the proof is complete.

\section{Proof of Lemma~\ref{lem:Dphi:rep}}
\label{app:Dphi:rep:proof}
Let $\ks = (k_1,\dots,k_d)\in \nats^d$ be a multi-index.  We have 
\begin{align*}
	P(\Xbs^n \mid \phi = k) &= 
	\sum_{\ks \in \nats^d}
	P(\Xbs^n \mid \lams = \ks) 
	\,\pr(\lams = \ks \mid \phi = k)
	\\
	&= \sum_{\ks \in \nats^d}
	\Big\{ 	\prod_{e \,\in\, \Et} P(\Xb_e^n \mid \lam_e = k_e)\Big\} \pipk(\ks)
\end{align*}
where we have used the extended edge notation of Section~\ref{sec:optimal:delay:thm} and conditional distribution introduced in~\eqref{eq:pipk:def}. Using the pre- and post-change densities, we get
\begin{align}
	P(\Xbs^n \mid \phi = k) 
	&= \sum_{\ks \in \nats^d}
	\Big\{ \prod_{e \,\in\,\Et} 
	\Big[\prod_{t=1}^{k_e-1} g_e(X_e^t)
	\prod_{t=k_e}^{n} f_e(X_e^t) \Big]
	\Big\} \pipk(\ks) \label{eq:Dphi:rep:temp}
\end{align}
where by convention, empty products are equal to $1$. Dividing~\eqref{eq:Dphi:rep:temp} by
\begin{align*}
	U(\Xbs^n) := \prod_{e \in \Et} \Big[\prod_{t=1}^n g_e(X_e^t) \Big].
\end{align*}
we obtain
\begin{align*}
	\frac{P(\Xbs^n \mid \phi = k)}{U(\Xbs^n)} = 
	\sum_{\ks \in \nats^d}
	\Big\{ \prod_{e \,\in\,\Et} R_{k_e}^n(\Xb_e) 
	\Big\} \pipk(\ks) = \mixR_\phi^{k,n}
\end{align*}
where we have used definitions~\eqref{eq:R:def} and~\eqref{eq:mixR:def}. The same expression holds, if we replace $k$ with $\infty$. The result now follows from definition~\eqref{eq:Dphi:def} of $\Dphi^{k,n}$.

\section{Proof of Lemma~\ref{lem:sum:prod}}
The idea of the proof is to write the sum as the diagonal part of a higher dimensional one and then drop the restriction to the diagonal. Let us illustrate the idea first by proving~\eqref{eq:sum:prod:2}.

We can write
\begin{align*}
	\sum_{p \in S_1} e^{-\beta_1 p} F_1(p) H_1(p) 
	&= \sum_{p \in S_1} \sum_{q \in S_1}
	1\{ p = q\} e^{-\frac{\beta_1}{2}(p+q)} F_1(p) H_1(q) \\
	&\le \sum_{p \in S_1} \sum_{q \in S_1}
	e^{-\frac{\beta_1}{2}(p+q)} F_1(p) H_1(q) 
\end{align*}
The bound holds since the terms are nonnegative. Now, the RHS factors over $p$ and $q$ and we get~\eqref{eq:sum:prod:2}.

The idea for the proof of~\eqref{eq:sum:prod:1} is similar. For every pair $(i,j) \in \binom{[r]}{2}$, we introduce new versions of $k_i$ and $k_j$ so that the corresponding term $G_{ij}(k_i,k_j)$ involves the new variables. To be more precise, let $\kcal = \{\nui_1,\dots,\nui_K\}$ be an enumeration of the elements of $\binom{[r]}{2}$. To each element $\nui_\ell = (i,j) \in \kcal$ with $i < j$, associate variables $u_\ell^1$ and $u_\ell^2$, representing newer versions of $k_i$ and $k_j$. In other words, $u_\ell^1$ is the new version of $k_{\nui_\ell(1)}$.

This procedures introduces $2K$ extra variables. To each of the original $k_i$ variables, there corresponds exactly $r-1$ new versions. Letting
\begin{align*}
	T_0 := \sum_{\kb\, \in\, \Sb} \Big\{ 
    e^{- \betab^T \kb} \prod_{j=1}^r F_j(k_j)
     	\prod_{(i,j) \in \binom{[r]}{2}} G_{ij}(k_i,k_j)
    \Big\} 
\end{align*}
denote the LHS of~\eqref{eq:sum:prod:1}, we have
\begin{align*}
  \begin{split}
    T_0 \;=\;
    \sum_{\{k_j\}, \;\{(u_\ell^1, u_\ell^2)\}}
    \Big\{
    1\{ u_\ell^1 = k_{\nui_\ell(1)},\; u_\ell^2 =
     k_{\nui_\ell(2)},\; \ell \in [K]  \} \times
     \qquad \qquad \\
    \exp \Big[ - \sum_{j} 
    	\frac{\beta_j}{r} k_j - \sum_{\ell \in [K]} 
    	\big( 
    	\frac{\beta_{\nui_\ell(1)}}{r} u_\ell^1+
    	\frac{\beta_{\nui_\ell(2)}}{r} u_\ell^2
    	\big)
    	\Big]
   	\prod_j F_j(k_j) 
   	\prod_{\ell \,\in\, \kcal}
   	G_{\nui_\ell}(u_\ell^1,u_\ell^2)
      \Big\}
  \end{split}
  \end{align*}
  where the summation is over $\{ k_j \in S_j, \; u_\ell^1 \in S_{\nui_\ell(1)}, \;u_\ell^2 \in S_{\nui_\ell(2)}, \;j \in [r], \;\ell \in [K]\}$. Dropping the indicator, we get an upper bound which separates
  \begin{align*}
  \begin{split}
    T_0 \;&\le\;
    \sum_{\{k_j\}, \;\{(u_\ell^1, u_\ell^2)\}}
    \Big\{
    e^{ -\sum_{j} 
    	\frac{\beta_j}{r} k_j - \sum_{\ell \in [K]} 
    	\big( 
    	\frac{\beta_{\nui_\ell(1)}}{r} u_\ell^1+
    	\frac{\beta_{\nui_\ell(2)}}{r} u_\ell^2
    	\big)
    	}
   	\prod_j F_j(k_j) 
   	\prod_{\ell \,\in\, \kcal}
   	G_{\nui_\ell}(u_\ell^1,u_\ell^2)
      \Big\} \\
     &=
     \prod_j \Big\{ \sum_{k_j} 
     e^{-\frac{\beta_j}{r}k_j} F_j(k_j)\Big\}
     \prod_{\ell \in \kcal} \Big\{
     	\sum_{(u_\ell^2,u_\ell^2)} 
     	e^{- \big( 
    	\frac{\beta_{\nui_\ell(1)}}{r} u_\ell^1+
    	\frac{\beta_{\nui_\ell(2)}}{r} u_\ell^2
    	\big)} G_{\nui_\ell}(u_\ell^1,u_\ell^2)
     \Big\}
  \end{split}
  \end{align*}
  which is the desired result.

\section{Proof of Lemma~\ref{lem:suff:cond}}
\label{app:proof:suff:cond}
   (a) We start by proving~(\ref{eq:suff:cond:up}). Pick $n_0$ large enough so that for $n \ge n_0$, we have
$q(n) \le C n^{a}$ for some numerical constant $a \in \nats$. Fix some
$k \in \nats$ throughout the proof. For now, fix $\ms \in \nats^d$
such that $\pipk(\ms) > 0$. Pick $\eps_n :=
\sqrt{\frac{\gamma \log n} {c_1 n}} \wedge \eps_0$ for some $\gamma$ to be determined
shortly. Note that $ \sqrt{n} \ge \frac{1}{\eps_n} p(\ms,k)$
is equivalent to $\sqrt{\frac{\gamma}{c_1}\log n} \wedge 
(\eps_0 \sqrt{n})\ge p(\ms,k)$,
which holds for sufficiently large $n$. Let $n_1 := n_1(\ms,k)$ be the
smallest $n$ for which this inequality holds.

 Using
the shorthand notation $\Dphi^{k,n} = \Dpk(\Xbs^n)$, we have for $n
\ge \max\{n_0,n_1\}$,
\begin{align*}
  \pr_{\lams}^{\ms} \Big\{ \Big| \frac1n \log \Dphi^{k,n} -
  I_\phi\Big| > \eps_n \Big\} \le C n^a \exp (-c_1 n\eps_n^2) = C n^{a
  - \gamma}
\end{align*}
Taking $\gamma = a+2$, we have by Borel-Cantelli lemma that
$\pr_{\lams}^{\ms} \{ \frac1n \log D_\phi^{k,n} \stackrel{n \to \infty}{\longrightarrow} I_\phi \} =
1$.  It follows that the sequence
$\{\frac{1}{n+k}\log D_\phi^{k,k+n}\}_n$ has the same limit
a.s. $\pr_{\lams}^{\ms}$. Since $k$ is fixed for now, $\frac{n}{n+k}
\sim 1$ as $n \to \infty$, hence $\pr_{\lams}^{\ms} \{ \frac1n \log
D_\phi^{k,k+n} \stackrel{n \to \infty}{\longrightarrow} I_\phi\} =
1$. 

We now take the average with respect to conditioned distribution
of $\lams$ given $\phi = k$. That is, we multiply by $\pipk(\ms)$ and
sum over $\ms$ to obtain
\begin{align}\label{eq:as:conv:under:pipk}
  \pr_{\phi}^{k} \Big\{ \frac1n \log
D_\phi^{k,k+n} \stackrel{n \to \infty}{\longrightarrow} I_\phi \Big\} =
1.
\end{align}
For any sequence of number $\{ b_n\}_{n \in \nats}$, $\frac{1}{n} b_n
\stackrel{n \to \infty}{\longrightarrow} b$ implies that%
\footnote{Here is the proof. Fix $\eps \in(0,b/2)$ and pick $n_0$ so that for
  $n \ge n_0$, $|\frac1n b_n - b| \le \eps$. Let $B_p^q := \frac1N
  \max_{p \le n \le q} b_n$. We have $B_1^N = \max \{ B_1^{n_0-1},
  B_{n_0}^N \}.$ We can pick $N_0$ such that for all $N \ge N_0$,
  $B_1^{n_0 - 1} \le \eps$. On the other hand, $n(b - \eps)\le  b_n
  \le n(b+\eps)$, for $n \in [n_0,N]$. Taking the maximum of each side
  over this
  interval, we obtain $N(b - \eps) \le N B_{n_0}^N \le N(b +
  \eps)$. Since $2\eps < b$, we have $B_1^N = B_{n_0}^N$ and
  $|B_{n_0}^N - b | \le \eps$ which implies the result.}%
$ \frac1 N \max_{1 \le n \le N} b_n \stackrel{N \to
  \infty}{\longrightarrow} b$. Thus, it follows
from~(\ref{eq:as:conv:under:pipk}) that
\begin{align}
  \pr_{\phi}^{k} \Big\{ \frac1N \max_{1 \le n \le N}\log
D_\phi^{k,k+n} \stackrel{N \to \infty}{\longrightarrow} I_\phi \Big\} =
1.
\end{align}
Since convergence a.s. implies convergence in probability, this
implies~(\ref{eq:suff:cond:up}).

\bigskip
(b) To prove~(\ref{eq:suff:cond:down}), let us fix $\eps \in(0,\eps_0)$ throughout. Changing $n$ to $n-k+1$ in the
definition of $T_\eps^k$, we obtain
\begin{align*}
  T_\eps^k &= -k+1 + \sup \big\{ n \ge k:\; \frac{1}{n-k+1} \log \Dphi^{k,n} <
  I_\phi - \eps \big\} \\
  &= -k+1 + \sup \big\{ n \ge k:\; \frac{1}{n} \log \Dphi^{k,n} <
  \frac{n-k+1}{n} (I_\phi - \eps) \big\} \\
  &\le -k+1 + \underbrace{\sup \{ n \ge 1:\; \frac1n \log \Dphi^{k,n} < I_\phi -
  \eps \big\}}_{\Tt_\eps^k}.
\end{align*}
Thus, it is enough to verify~(\ref{eq:suff:cond:down}) for
$\Tt_\eps^k$ in place of $T_\eps^k$.

 Let $Y^{k,n} := \frac{1}{n}
\log \Dphi^{k,n}$. For $\ms \in \nats^d$, let $n_2 := n_2(k,m;\eps)$ be
the smallest integer $n$ that satisfies $\sqrt{n} \ge
\frac1\eps p(\ms,k)$, that is
\begin{align}\label{eq:n2:bound}
  n_2(k,m;\eps) := \ceil{\frac{1}{\eps^2} p^2(\ms,k)} \le 
 \frac{1}{\eps^2} p^2(\ms,k) + 1.
\end{align}

 Let $n_0$ be as in the previous part. By
assumption, for all $n \ge \max\{n_2,n_0\}$, we have $\pr_{\lams}^{\ms}
\{ |Y^{n,k} - I_\phi| > \eps \} \le C n^{a} \exp(-n\eps^2)$. To
simplify notation, we will assume $n_0 = 1$ without loss of
generality. We have
\begin{align*}
  \ex_{\lams}^{\ms} [\Tt_\eps^k] = \sum_{\ell = 1}^\infty
  \pr_{\lams}^{\ms} ( \Tt_\eps^k > \ell) &= 
  \sum_{\ell \ge 1} \pr_{\lams}^{\ms} \Big(\bigcup_{n > \ell}
  \{Y^{n,k} <  I_\phi -\eps\}  \Big) \\
  &\le \sum_{\ell \ge 1} \sum_{n > \ell} \pr_{\lams}^{\ms}  (
  Y^{n,k} < I_\phi -\eps) \\
  &= \sum_{n=1}^\infty (n-1) \pr_{\lams}^{\ms}  (
  Y^{n,k} < I_\phi -\eps) \\
  &\le \sum_{n=1}^{n_2 -1} (n-1) + \sum_{n \ge n_2} C n^{a} e^{-n
    \eps^2} \\
  &\le \frac{(n_2-1)^2}{2} + \sum_{n =1}^\infty C n^{a} e^{-n
    \eps^2}.
\end{align*}
The second term on the RHS does not depend on $\ms$ or $k$, and we can
denote it as $C_1(\eps)$. Using the bound~(\ref{eq:n2:bound}) on
$n_2$, we have
\begin{align*}
  \ex_\phi^k [ \Tt_\eps^k ] \le \frac{1}{2 \eps^4} \sum_{\ms \in \nats
  }
  \pipk(\ms) p^4(\ms,k) + C_1(\eps).
\end{align*}
Since by assumption, both $\pipk(\cdot)$ and $\pr(\phi = \cdot)$ have
finite polynomial moments, it follows that~(\ref{eq:suff:cond:down})
holds for $\Tt_\eps^k$.

\bibliographystyle{IEEEtran}
\bibliography{mcp_it12,sensorbib}

\end{document}